\documentclass{amsart}

\usepackage{amsmath, amssymb, amsthm, mathrsfs, graphicx, combelow}
\usepackage[backref=page]{hyperref}
\usepackage{enumitem, xspace, ifthen, comment}
\usepackage[all]{xy}
\usepackage{tikz}
\usetikzlibrary{arrows,shapes,trees}

\newtheorem*{thm-plain}{Theorem}
\newtheorem{thm}{Theorem}[section]
\newtheorem{lem}[thm]{Lemma}
\newtheorem{prp}[thm]{Proposition}
\newtheorem{cor}[thm]{Corollary}

\newtheorem{conj}[thm]{Conjecture}
\newtheorem{ques}[thm]{Question}
\numberwithin{equation}{thm}

\theoremstyle{definition}
\newtheorem{dfn}[thm]{Definition}

\newtheorem*{dfn-plain}{Definition}

\theoremstyle{remark}
\newtheorem{clm}[thm]{Claim}

\newtheorem{awlog}[thm]{Additional Assumption}

\newtheorem{rem}[thm]{Remark}

\newtheorem{exm}[thm]{Example}
\newtheorem*{rem-plain}{Remark}

\DeclareMathOperator{\Spec}{Spec}
\DeclareMathOperator{\trdeg}{trdeg}
\DeclareMathOperator{\Aut}{Aut}
\DeclareMathOperator{\Autn}{Aut^\circ}
\DeclareMathOperator{\Pic}{Pic}
\DeclareMathOperator{\codim}{codim}

\def\rd#1.{\lfloor{#1}\rfloor}
\def\rp#1.{\lceil{#1}\rceil}

\def\tw#1.{\langle{#1}\rangle}

\newcommand{\imp}{\Rightarrow}
\newcommand{\lto}{\longrightarrow}

\newcommand{\N}{\mathbb N}
\newcommand{\Z}{\mathbb Z}
\newcommand{\Q}{\ensuremath{\mathbb Q}}
\newcommand{\R}{\mathbb R}
\newcommand{\C}{\mathbb C}

\renewcommand{\P}{\mathbb P}
\renewcommand{\O}{\mathscr O}

\newcommand{\x}{\times}

\renewcommand{\i}{\mathrm{i}}
\renewcommand{\phi}{\varphi}
\renewcommand{\theta}{\vartheta}
\newcommand{\id}{\mathrm{id}}

\newcommand{\minus}{\setminus}%{\backslash}
\newcommand{\inj}{\hookrightarrow}
\newcommand{\surj}{\twoheadrightarrow}
\newcommand{\bij}{\xrightarrow{\;\sim\;}}
\newcommand{\isom}{\cong}
\newcommand{\mf}{\mathfrak}

\newcommand{\tensor}{\otimes}

\DeclareMathOperator{\Hom}{\mathrm{Hom}}
\DeclareMathOperator{\Isom}{\mathrm{Isom}}
\DeclareMathOperator{\sHom}{\mathscr H\!om}
\DeclareMathOperator{\Ext}{Ext}
\DeclareMathOperator{\sExt}{\mathscr E\!xt}
\newcommand{\sg}{\mathrm{sg}}
\newcommand{\sm}{\mathrm{sm}}

\newcommand{\an}{\mathit{an}}

\DeclareMathOperator{\pr}{pr}

\newcommand{\cy}{Calabi--Yau\xspace}
\newcommand{\kahler}{K\"ahler\xspace}
\DeclareMathOperator{\Alb}{Alb}
\newcommand{\alb}{\mathrm{alb}}
\DeclareMathOperator{\Def}{Def}
\DeclareMathOperator{\Deflt}{Def^{lt}}
\DeclareMathOperator{\tr}{tr}
\DeclareMathOperator{\Der}{Der_\C}

\newcommand{\sA}{\mathscr{A}}

\newcommand{\sC}{\mathscr{C}}
\newcommand{\cC}{\mathcal{C}}

\newcommand{\sF}{\mathscr{F}}

\newcommand{\cH}{\mathcal{H}}

\newcommand{\cK}{\mathcal{K}}

\newcommand{\sM}{\mathscr{M}}

\newcommand{\sT}{\mathscr{T}}

\newcommand{\cU}{\mathcal{U}}

\newcommand{\cV}{\mathcal{V}}

\newcommand{\sX}{\mathscr{X}}
\newcommand{\sY}{\mathscr{Y}}

\newdir{ ir}{{}*!/-5pt/@^{(}} % injektiv (->
\newdir{ il}{{}*!/-5pt/@_{(}} % injektiv <-)

\DeclareMathOperator{\supp}{supp}

\newcommand{\eps}{\varepsilon}

\renewcommand{\d}{\mathrm d}
\newcommand{\del}{\partial}
\newcommand{\wt}{\widetilde}
\newcommand{\wh}{\widehat}

\DeclareMathOperator{\rk}{rk}

\newenvironment{sequation}{%
\setcounter{equation}{\value{thm}}%
\numberwithin{equation}{section}%
\begin{equation}%
}{%
\end{equation}%
\numberwithin{equation}{thm}%
\addtocounter{thm}{1}%
}

\numberwithin{equation}{thm}

\DeclareRobustCommand{\SkipTocEntry}[5]{}

\setitemize[1]{leftmargin=*,parsep=0em,itemsep=0.125em,topsep=0.125em}
\setenumerate[1]{leftmargin=*,parsep=0em,itemsep=0.125em,topsep=0.125em}

\newcommand{\iref}[3]{\the\value{#1}.\the\value{#2}(\the\value{#3})}

\newcommand\factor[2]{\left. \raise 2pt\hbox{$#1$} \right/\hskip -2pt \raise -2pt\hbox{$#2$}}

\definecolor{forrest}{RGB}{81,133,49}
\definecolor{mydarkblue}{RGB}{10,92,153}

\newcommand{\PreprintAndPublication}[2]{#1}

\title[Algebraic approximation of K\"ahler threefolds with $\kappa = 0$]{Algebraic approximation of K\"ahler threefolds of Kodaira dimension zero}

\dedicatory{To sing, to brag, to blaze, to grumble \dots}

\author{Patrick Graf}
\address{Lehrstuhl f\"ur Mathematik I, Universit\"at Bayreuth, 95440 Bayreuth, Germany}
\email{\href{mailto:patrick.graf@uni-bayreuth.de}{patrick.graf@uni-bayreuth.de}}
\urladdr{\href{http://www.pgraf.uni-bayreuth.de/en/}{www.graficland.uni-bayreuth.de}}

\date{January 17, 2018}
\thanks{The author was partially supported by the DFG grant ``Zur Positivit\"at in der komplexen Geometrie''.} %

\hypersetup{
  pdfauthor={Patrick Graf},
  pdftitle={Algebraic approximation of Kahler threefolds of Kodaira dimension zero},
  pdfkeywords={Kahler manifolds, Calabi-Yau threefolds, algebraic approximation, small projective deformations, locally trivial deformations, Kahler groups, Albanese map},
  pdfstartview={Fit},
  pdfpagelayout={OneColumn},
  pdfpagemode={UseNone},
  linktoc=all,
  breaklinks,
  linkcolor=[RGB]{0 0 96},
  citecolor=[RGB]{96 0 0},
  urlcolor=[RGB]{0 96 0},
  colorlinks}

\begin{document}

\begin{abstract}
We prove that for a compact K\"ahler threefold with canonical singularities and vanishing first Chern class, the projective fibres are dense in the semiuniversal deformation space.
This implies that every K\"ahler threefold of Kodaira dimension zero admits small projective deformations after a suitable bimeromorphic modification.
As a corollary, we see that the fundamental group of any K\"ahler threefold is a quotient of an extension of fundamental groups of projective manifolds, up to subgroups of finite index.

In the course of the proof, we show that for a canonical threefold with $c_1 = 0$, the Albanese map decomposes as a product after a finite \'etale base change.
This generalizes a result of Kawamata, valid in all dimensions, to the K\"ahler case.
Furthermore we generalize a Hodge-theoretic criterion for algebraic approximability, due to Green and Voisin, to quotients of a manifold by a finite group.
\end{abstract}

\maketitle

\begingroup
\hypersetup{linkcolor=black}
\tableofcontents
\endgroup

\section{Introduction}

A fundamental problem in complex algebraic and \kahler geometry is to determine the relationship between smooth projective varieties and compact \kahler manifolds.
The famous Kodaira Embedding Theorem, combined with Chow's theorem, says that a compact complex manifold is projective if and only if it admits a \kahler form whose cohomology class is rational.
This suggests the following question, generally attributed to Kodaira.

\begin{ques}[Kodaira problem] \label{ques:Kodaira}
Is it possible to make any compact \kahler manifold $X$ projective by an arbitrarily small deformation $X_t$ of its complex structure?
\end{ques}

Such a deformation will be called an \emph{algebraic approximation} of $X$ (see Definition~\ref{dfn:alg approx}).

Using his classification of surfaces, Kodaira proved that every compact \kahler surface can be deformed to an algebraic surface~\cite[Thm.~16.1]{Kod63}.
More recently Buchdahl gave a classification-free proof, showing that indeed the deformation can be taken arbitrarily small~\cite{Buc06, Buc08}.
In higher dimensions, the Kodaira problem remained open until in 2003 Voisin \cite{Voi04} gave counterexamples in any dimension $\ge 4$. Her examples are bimeromorphic to a torus, which does in fact have an algebraic approximation.
But in 2004 Voisin~\cite{Voi06} even constructed examples (in any even dimension $\ge 10$) of compact \kahler manifolds $X$ such that no compact complex manifold $X'$ bimeromorphic to $X$ admits an algebraic approximation.

At first sight, this seems to provide a most definite negative answer to the Kodaira problem. However, from the viewpoint of the Minimal Model Program (MMP), it is more natural to ask about algebraic approximation for minimal models of compact \kahler manifolds. The point is that these spaces will in general be singular, albeit mildly.
To be more precise, we say that a normal compact \kahler space $X$ is \emph{minimal} if it is \Q-factorial with at worst terminal singularities, and $K_X$ is nef.
(See Sections~\ref{sec:kahler spaces} and~\ref{sec:MMP sings} for the notion of singular \kahler space and the other terms used in this definition.)
The following version of the Kodaira problem for minimal spaces was proposed by Thomas Peternell.

\begin{conj}[Peternell] \label{conj:Peternell}
Any minimal \kahler space admits an algebraic approximation.
In particular, assuming the \kahler MMP works, every non-uniruled compact \kahler manifold $X$ is bimeromorphic to a (mildly singular) compact \kahler space $X'$ which admits an algebraic approximation.
\end{conj}

By ``assuming that the \kahler MMP works'' we mean that every non-uniruled compact \kahler manifold should be bimeromorphic to a minimal \kahler space.

\begin{rem-plain}[Surfaces]
Conjecture~\ref{conj:Peternell} fits in nicely with Kodaira's result about approximation of surfaces.
Namely, since terminal surface singularities are smooth and approximability is preserved under blowing up a smooth point, approximation of (non-ruled) surfaces follows from Conjecture~\ref{conj:Peternell}.
\end{rem-plain}

\begin{rem-plain}[Uniruled manifolds]
It is not quite obvious what conjecture to formulate for uniruled manifolds.
At least one cannot hope for approximability for the total space of any Mori fibration.
Indeed, let $X$ be a rigid non-projective compact \kahler manifold, as for example provided by~\cite{Voi04}. Then $X \x \P^1 \to X$ is a Mori fibre space, but $X \x \P^1$ does not admit an algebraic approximation since it is likewise rigid and non-projective.
\end{rem-plain}

\subsection{Main results}

The main result of this paper is the confirmation of Conjecture~\ref{conj:Peternell} for $X$ a threefold of Kodaira dimension zero.
In fact, we prove the following stronger result, which is not empty even if $X$ itself is already projective.
We refer to Section~\ref{sec:deformation theory} for the definition of the miniversal locally trivial deformation and of a strong algebraic approximation.

\begin{thm} \label{thm:bim approx intro}
Let $X$ be a compact \kahler manifold with $\dim X = 3$ and $\kappa(X) = 0$. Then there is a bimeromorphic modification $X \dashrightarrow X'$, where $X'$ is a normal compact \kahler space with \Q-factorial terminal singularities, such that the miniversal locally trivial deformation $\sX' \to \Deflt(X')$ of $X'$ is a strong algebraic approximation of $X'$.
\end{thm}

The MMP predicts that the basic building blocks of \kahler manifolds are rationally connected varieties, manifolds of Kodaira dimension zero, and varieties of general type. The former and the latter are automatically algebraic.
Hence it is natural to consider the case $\kappa = 0$ first.
We will come back to this observation in Corollary~\ref{cor:kahler groups}, where we study the fundamental groups of arbitrary \kahler threefolds.

Theorem~\ref{thm:bim approx intro} is an immediate consequence of Theorem~\ref{thm:approx intro} below, combined with the MMP for \kahler threefolds as established recently by H\"oring--Peternell and Campana--H\"oring--Peternell~\cite{HP15, CHP15}.

\begin{thm} \label{thm:approx intro}
Let $X$ be a \cy threefold in the sense of Definition~\ref{dfn:cy} below. Then:
\begin{enumerate}
\item\label{itm:approx intro.1} The miniversal deformation $\sX \to \Def(X)$ of $X$ is a strong algebraic ap\-proximation of $X$.
\item\label{itm:approx intro.2} If the singularities of $X$ are isolated, then also the miniversal locally trivial deformation $\sX' \to \Deflt(X)$ of $X$ is a strong algebraic approximation.
\end{enumerate}
\end{thm}

Here we use the following definition of a \cy threefold. Note that there are several competing definitions in the literature, differing both in whether projectivity is required and in the kind of singularities allowed.

\begin{dfn}[\cy space] \label{dfn:cy}
A \emph{\cy space} is a normal compact \kahler space $X$ with canonical singularities and vanishing first Chern class $c_1(X) = 0 \in H^2(X, \R)$.
A \emph{\cy threefold} is a \cy space of dimension three.
\end{dfn}

\begin{rem-plain}[Deformations of Calabi--Yaus]
The property of being \cy in this sense is stable under small deformations (Lemma~\ref{lem:cy deformation}).
\end{rem-plain}

\begin{rem-plain}[Smooth Calabi--Yaus]
For compact \kahler \emph{manifolds} with $c_1 = 0$, algebraic approximability is well-known in any dimension. See e.g.~\cite[Prop.~4.2.2]{Cao13} for the argument.
However, as pointed out above, in general the vanishing $c_1 = 0$ can only be achieved at the cost of introducing terminal singularities and hence one needs to deal at least with these.
\end{rem-plain}

\subsection{Corollaries: fundamental groups of \kahler manifolds}

A general problem regarding the topology of \kahler manifolds is to understand which finitely presented groups appear as the fundamental group of a compact \kahler manifold. These groups are called \emph{\kahler groups} for short. Similarly, we say that a group is \emph{(complex-)projective} if it is the fundamental group of the analytification $X^\an$ of a smooth projective variety $X$ defined over $\C$.
Curiously, all known examples of \kahler groups are complex-projective, while most known restrictions on complex-projective groups also apply to \kahler groups.
This leads to the following long-standing open problem~\cite[(1.26)]{ABCKT96}.

\begin{ques}
Is every \kahler group the fundamental group of a complex projective manifold?
\end{ques}

We give an affirmative partial answer in dimension three: the fundamental group of any \kahler threefold can be built from projective groups by taking quotients and extensions.
Furthermore we give a complete affirmative answer for \kahler threefolds of Kodaira dimension zero.

\begin{cor}[Fundamental groups of \kahler threefolds, I] \label{cor:kahler groups}
Let $X$ be a compact \kahler threefold. Then there exist complex projective manifolds $F$ and $Z$ and exact sequences of groups
\[ \pi_1(F) \lto G \lto \pi_1(Z) \lto 1, \]
\[ G \lto \pi_1(X) \lto Q \lto 1, \]
where $Q$ is finite (hence in particular projective).
\end{cor}

\begin{cor}[Fundamental groups of \kahler threefolds, II] \label{cor:pi1}
Let $X$ be a compact \kahler threefold with $\kappa(X) = 0$. Then:
\begin{enumerate}
\item There exists a smooth projective threefold $Y$ with $\kappa(Y) = 0$ such that $\pi_1(X) \isom \pi_1(Y)$. In particular, $\pi_1(X)$ is almost abelian, i.e.~it contains an abelian normal subgroup $\Gamma \trianglelefteq \pi_1(X)$ of finite index.
\item All finite index abelian subgroups of $\pi_1(X)$ have rank $2 \cdot \wt q(X)$, where $0 \le \wt q(X) \le 3$ is the augmented irregularity of $X$ (Definition~\ref{dfn:irregularity}).
\item In particular, if $\wt q(X) = 0$ then $\pi_1(X)$ is finite.
\end{enumerate}
\end{cor}

\begin{rem-plain}[Work of Campana and Claudon]
In~\cite[Thm.~1.1]{CC14} Campana and Claudon have already proved the almost abelianity part of Corollary~\ref{cor:pi1}, without using the Minimal Model Program for \kahler threefolds.
\end{rem-plain}

A \cy threefold $X$ with $\wt q(X) = 0$ is necessarily projective (Remark~\ref{rem:wt q = 0}).
The following finiteness criterion for $\pi_1(X)$ also applies to certain non-projective \cy threefolds.
For projective varieties, it was proven in all dimensions in~\cite[Prop.~8.23]{BBdecomp}.

\begin{cor}[Fundamental groups of \cy threefolds] \label{cor:pi1 cy3}
Let $X$ be a \cy threefold. If $\chi(X, \O_X) \ne 0$, then $\pi_1(X)$ is finite, of cardinality
\[ |\pi_1(X)| \le \frac 4 {|\chi(X, \O_X)|}. \]
\end{cor}

\begin{rem-plain}[Threefolds with $\chi \ne 0$]
If $X$ is a \cy threefold with trivial canonical bundle, then $\chi(X, \O_X) = 0$ by Serre duality.
But for a non-smooth, non-projective \cy threefold $X$ with $\omega_X \not\isom \O_X$, we have $\chi(X, \O_X) \ge 1$ (Lemma~\ref{lem:chi ne 0}).
\end{rem-plain}

\begin{rem-plain}[Orbifold fundamental groups]
Corollary~\ref{cor:kahler groups} is obtained by studying the pluricanonical map $\phi = \phi_{|mK_X|}\!: X \dashrightarrow Z$ of $X$.
A difficulty in the proof is that we have no control over the multiple fibres of $\phi$.
What we do have is an exact sequence
\[ \pi_1(F) \lto \pi_1(X) \lto \pi_1(Z, \Delta) \lto 1 \]
of orbifold fundamental groups~\cite[Prop.~11.9]{Cam09}. Here $\Delta$ is a suitable divisor on $Z$ encoding the singularities of $\phi$.
However, the class of projective orbifold fundamental groups is somewhat larger than the class of complex-projective groups.
For threefolds $X$ that satisfy the (unrealistic) assumption that the pluricanonical map is smooth, the above exact sequence simplifies to
\[ \pi_1(F) \lto \pi_1(X) \lto \pi_1(Z) \lto 1. \]
\end{rem-plain}

\PreprintAndPublication{
\subsection{Corollaries: Unobstructedness and smoothability} \label{sec:unobs sm}

The miniversal deformation space of a \cy manifold is smooth by the famous Bogomolov--Tian--Todorov theorem~\cite[Cor.~2]{Ran92}.
As soon as singularities are allowed, the situation becomes more complicated.
Gross~\cite[Ex.~2.7]{Gro97b} gave an example of a projective \cy threefold with isolated canonical singularities whose deformation space is non-reduced.
On the other hand, Namikawa~\cite[Thm.~A]{Nam94} showed that $\Def(X)$ is smooth for a projective \cy threefold with terminal singularities.
We extend Namikawa's result to the \kahler case, while showing at the same time that Gross' example does not exist in the non-projective setting.

\begin{cor}[Unobstructedness of \cy threefolds] \label{cor:unobs}
Let $X$ be a \cy threefold. Assume that either
\begin{enumerate}
\item $X$ has terminal singularities, or
\item $X$ is non-projective and has isolated canonical singularities.
\end{enumerate}
Then $\Def(X)$ is smooth.
\end{cor}

A closely related question is whether a singular Calabi--Yau admits a smoothing, i.e.~a flat deformation whose general fibre is smooth.
The following positive result was proved by Namikawa and Steenbrink~\cite[Thm.~1.3]{NS95} in the projective case.

\begin{cor}[Smoothability of \cy threefolds] \label{cor:smoothability}
Let $X$ be a \Q-factorial \cy threefold with only isolated canonical hypersurface singularities and $\omega_X \isom \O_X$. Then $X$ has a smoothing.
\end{cor}

Regarding higher-dimensional \cy spaces, we still have the following partial result.
It applies for example to Calabi--Yaus whose index one cover is smooth.

\begin{cor}[Quotients of smooth Calabi--Yaus] \label{cor:quot cy}
Let $X$ be a \cy manifold and $G$ a finite group acting on $X$ freely in codimension one.
Then the quotient space $Y = X/G$ has a smooth locally trivial deformation space $\Deflt(Y)$, and the miniversal family $\mf Y \to \Deflt(Y)$ is a strong algebraic approximation of $Y$.
\end{cor}
}{}

\subsection{Further results}

In the course of the proof of Theorem~\ref{thm:approx intro}, we obtain two results which are of independent interest.
The first one concerns the structure of the Albanese map of a \cy threefold.
For projective varieties, this was proven in arbitrary dimensions by Kawamata~\cite[Thm.~8.3]{Kaw85}.

\begin{thm}[Albanese map of a \cy threefold] \label{thm:alb cy3 intro}
Let $X$ be a \cy threefold, with Albanese map $\alpha\!: X \to A := \Alb(X)$.
Then there exists a finite \'etale cover $A_1 \to A$ such that $X \x_A A_1$ is isomorphic to $F \x A_1$ over $A_1$, where $F$ is connected.
In particular, $\alpha$ is a surjective analytic fibre bundle with connected fibre.
\end{thm}

The second result is a Hodge-theoretic criterion for algebraic approximability.
Let $Y$ be a compact \kahler manifold.
To any \kahler class $[\omega] \in H^{1,1}(Y) = H^1(Y, \Omega_Y^1)$, one can associate a linear map
\[ \alpha_{[\omega]}\!: H^1(Y, \sT_Y) \xrightarrow{[\omega] \cup -} H^2(Y, \Omega_Y^1 \tensor \sT_Y) \xrightarrow{\text{contraction}} H^2(Y, \O_Y). \]
A theorem of Green and Voisin~\cite[Prop.~5.20]{Voi03} says that if $Y$ is unobstructed (meaning that $\Def(Y)$ is smooth) and $\alpha_{[\omega]}$ is surjective for some \kahler class $[\omega]$, then $Y$ has a strong algebraic approximation.
Our criterion is a generalization of this result to quotients of $Y$ by a finite group.

\begin{thm}[Approximation criterion for quotient spaces] \label{thm:approx crit quotient intro}
Let $Y$ be an unobstructed compact \kahler manifold, and let $G$ be a finite group acting on $Y$.
Assume that for some $G$-invariant \kahler class $[\omega]$ on $Y$, the map $\alpha_{[\omega]}$ defined above is surjective. Then:
\begin{enumerate}
\item\label{itm:acqi.1} The quotient space $X = Y/G$ admits a locally trivial strong algebraic approximation.
\item\label{itm:acqi.2} If the $G$-action on $Y$ is free in codimension one, then $\Deflt(X)$ is smooth and the miniversal locally trivial deformation $\mf X \to \Deflt(X)$ is a strong algebraic approximation of $X$.
\end{enumerate}
\end{thm}

\begin{rem-plain}[Generalizations]
It seems reasonable to expect that Theorem~\ref{thm:approx crit quotient intro} can be generalized further to, say, locally trivial deformations of spaces with rational singularities. However, the present version is sufficient for our purposes.
\end{rem-plain}

\PreprintAndPublication{
\subsection{Outline of proof of Theorem~\ref{thm:approx intro}}

The proof of Theorem~\ref{thm:approx intro} proceeds in three main steps. Let $X$ be a \cy threefold. We wish to show that $X$ admits an algebraic approximation.

\subsubsection*{Step 1: $X$ has isolated singularities and trivial canonical sheaf}

In this case, Theorem~\ref{thm:alb cy3 intro} on the Albanese map of $X$, combined with Kodaira's Embedding Theorem, implies rather quickly that if $X$ is non-projective then it is smooth (Theorem~\ref{thm:cy3 K=0}). For smooth Calabi--Yaus, approximability is known by a combination of the Bogomolov--Tian--Todorov theorem, the Hard Lefschetz theorem and the Green--Voisin criterion mentioned above.

\subsubsection*{Step 2: $X$ has isolated singularities}

By a standard construction known as the ``index one cover'', there is a finite Galois cover $X_1 \to X$ such that $X_1$ satisfies the assumptions of Step 1.
If $X_1$ is projective, then so is $X$. On the other hand, an algebraic approximation of $X_1$ descends to $X$ thanks to Theorem~\ref{thm:approx crit quotient intro}.
This proves~(\ref{thm:approx intro}.\ref{itm:approx intro.2}).

\subsubsection*{Step 3: General case}

We consider a ``crepant terminalization'' $Y \to X$, which is a partial resolution of $X$ such that $Y$ satisfies the assumptions of Step 2. By a general deformation theoretic result, an approximation of $Y$ blows down to one of $X$ thanks to the fact that $X$ has rational singularities.
This proves~(\ref{thm:approx intro}.\ref{itm:approx intro.1}) and finishes the proof of Theorem~\ref{thm:approx intro}.
}{}

\subsection{Earlier (and later) work}

We briefly review the current knowledge on algebraic approximation of \kahler manifolds.
As already mentioned, the case of surfaces is due to Kodaira~\cite{Kod63} and Buchdahl~\cite{Buc06, Buc08}, while negative results in dimension $\ge 4$ were given by Voisin~\cite{Voi04, Voi06}.
More recently, Schrack~\cite{Sch12} established approximability for certain kinds of $\P^1$-fibrations over \kahler surfaces.
In~\cite{Cao15}, Cao proved approximation for \kahler manifolds $X$ of arbitrary dimension under additional curvature assumptions, e.g.~$-K_X$ hermitian semipositive or $T_X$ nef.

Since this work was first published as a preprint, there has been substantial progress in dimension three.
It is now known that any \kahler threefold of Kodaira dimension $\kappa = 0$ or $1$ (not assumed to be minimal) admits an algebraic approximation, by the work of Lin~\cite{Lin16, Lin17a}.
If $\kappa = 2$, algebraic approximability at least of a suitable bimeromorphic model was established by Claudon and H\"oring~\cite{ClaudonHoeringKahlerGroups}.
Finally, Lin also completely settled the uniruled case, i.e.~$\kappa = -\infty$~\cite{Lin17b}.
In particular, it is now clear that $\pi_1$ of any compact \kahler threefold is isomorphic to $\pi_1$ of a projective manifold, a result hinted at by our Corollary~\ref{cor:kahler groups}.

In contrast, these questions remain wide open in dimensions $\ge 4$, even if we assume that the MMP works.

\subsection{Acknowledgements}

I would like to thank Thomas Peternell for suggesting the topic of algebraic approximation to me, and for many fruitful discussions.
Furthermore I would like to thank Zsolt Patakfalvi, Junyan Cao and in particular Christian Lehn for interesting conversations and for answering my questions.
Last but not least, the anonymous referee's input has been very helpful both in simplifying some of the arguments and in improving the statement of Theorem~\ref{thm:alb cy3 intro} significantly.

\section{Notation and standard facts}

\subsection{Complex spaces}

All complex spaces are assumed to be separated, connected and reduced, unless otherwise stated.

A compact complex manifold $X$ is said to be \emph{of Fujiki class $\cC$} (or \emph{in $\cC$}, for short) if it is bimeromorphic to a compact \kahler manifold.
We say that $X$ is \emph{Moishezon} if its field of meromorphic functions $\sM(X)$ has maximal transcendence degree $\trdeg_\C \sM(X) = \dim X$. Being Moishezon is equivalent to being bimeromorphic to a projective manifold.

We denote by $\Omega_X^1$ the sheaf of \kahler differentials on a normal complex space $X$.
The tangent sheaf $\sT_X$ of $X$ is by definition the dual $\sHom_{\O_X}(\Omega_X^1, \O_X)$. It is a reflexive sheaf on $X$. In particular it coincides with $i_* \sT_{X^\sm}$, where $i\!: X^\sm \to X$ is the inclusion of the smooth locus.

By an \emph{analytic fibre bundle}, we mean a proper surjective map $f\!: X \to Y$ of complex spaces such that every point $y \in Y$ has an open neighborhood $U \subset Y$ with $f^{-1}(U) \isom U \x f^{-1}(y)$ over $U$.
A proper map $f\!: X \to Y$ is called a \emph{fibre space} if it is surjective with connected fibres.

A \emph{resolution of singularities} of $X$ is a proper bimeromorphic morphism $f\!: \wt X \to X$, where $\wt X$ is smooth. We say that the resolution is \emph{projective} if $f$ is a projective morphism. In this case, if $X$ is compact \kahler then so is $\wt X$.
Any compact space $X$ has a projective resolution by~\cite[Thm.~3.45]{Kol07}.

We denote \Q-linear equivalence of Weil divisors by the symbol $\sim_\Q$.
If $X$ is compact and $\sF$ is a coherent sheaf on $X$, we write $h^i(X, \sF) := \dim_\C H^i(X, \sF)$.

\subsection{\kahler spaces} \label{sec:kahler spaces}

While we will not work directly with the definition of a singular \kahler space, we include the definition here for the reader's convenience.

\begin{dfn}[\kahler space] \label{dfn:kahler}
Let $X$ be a normal complex space. A \emph{\kahler form} $\omega$ on $X$ is a \kahler form $\omega^\circ$ on the smooth locus $X_\sm \subset X$ such that $X$ can be covered by open sets $U_\alpha$ with the following property: there is an embedding $U_\alpha \inj W_\alpha$ of $U_\alpha$ as an analytic subset of an open set $W_\alpha \subset \C^{n_\alpha}$ and a strictly plurisubharmonic $\sC^\infty$ function $f_\alpha\!: W_\alpha \to \R$ such that
\[ \omega^\circ\big|_{U_\alpha \cap X_\sm} = \big( \i\del\bar\del f_\alpha \big) \big|_{U_\alpha \cap X_\sm}. \]
A normal complex space $X$ is said to be \emph{\kahler} if there exists a \kahler form on $X$.
\end{dfn}

In particular, the analytification of a normal complex projective variety is a \kahler space.

\begin{dfn}[Nef line bundles]
A line bundle $L$ on a compact \kahler manifold $X$ is said to be \emph{nef} if $c_1(L)$ is contained in the closure of the open cone $\cK \subset H^{1,1}(X) \cap H^2(X, \R)$ generated by the cohomology classes of \kahler forms. For the technically correct definition of nefness on singular \kahler spaces, see~\cite[Def.~3.10]{HP15}.
\end{dfn}

\subsection{Singularities of the MMP} \label{sec:MMP sings}

For the definition of terminal, canonical and klt singularities we refer to~\cite[Def.~2.34]{KM98}. We have the following chain of implications:
\[ \text{terminal} \imp \text{canonical} \imp \text{klt} \imp \text{rational}. \]
The first two hold by definition and the third one is~\cite[Thm.~5.22]{KM98}.
For a klt space $X$, the canonical sheaf $\omega_X$ is \Q-Cartier, i.e.~some reflexive tensor power $\omega_X^{[m]} := \big( \omega_X^{\tensor m} \big)^{**}$ is locally free, where $m > 0$. In particular, the first Chern class of $X$ is well-defined. Namely, set
\[ c_1(X) := - \frac 1 m c_1 \big( \omega_X^{[m]} \big) \in H^2(X, \Q). \]
Although the canonical sheaf need not come from a Weil divisor, we will abuse notation and write $mK_X$ for $\omega_X^{[m]}$.

The normal complex space $X$ is said to be \emph{\Q-factorial} if for every Weil divisor $D$ on $X$ there exists an integer $m > 0$ such that the sheaf $\O_X(mD)$ is locally free, and additionally $\omega_X$ is \Q-Cartier.

\begin{dfn}[Quasi-\'etale maps]
A finite surjective map $f\!: X \to Y$ of normal complex spaces is said to be \emph{quasi-\'etale} if for some open subset $Y^\circ \subset Y$ with $\codim_Y (Y \minus Y^\circ) \ge 2$, the restriction of $f$ to $f^{-1}(Y^\circ) \to Y^\circ$ is \'etale.
\end{dfn}

\begin{dfn}[Irregularity] \label{dfn:irregularity}
The \emph{irregularity} of a compact complex space $X$ is $q(X) := h^1 \big( \wt X, \O_{\wt X} \big)$, where $\wt X \to X$ is a resolution of singularities.
The \emph{augmented irregularity} of $X$ is
\[ \wt q(X) := \max \big\{ q(\wt X) \;\big|\; \wt X \to X \text{ finite quasi-\'etale} \big\} \in \N_0 \cup \{ \infty \}. \]
\end{dfn}

\begin{rem}[Finiteness of $\wt q$] \label{rem:wt q finite}
If $X$ is a \cy threefold, then $X$ has rational singularities and by Theorem~\ref{thm:alb cy3 intro} the Albanese map of $X$ is surjective. Hence $q(X) \le 3$.
Moreover, any finite quasi-\'etale cover $\wt X$ of $X$ is again a \cy threefold~\cite[Prop.~5.20]{KM98}. Summarizing, we see that $\wt q(X) \le 3$.
\end{rem}

\begin{rem}[Spaces with $\wt q = 0$] \label{rem:wt q = 0}
A \cy threefold $X$ with $\wt q(X) = 0$ is necessarily projective. Namely, the index one cover $X_1$ of $X$~\cite[Def.~5.19]{KM98} satisfies $q(X_1) = 0$ and hence it is projective by the argument in the proof of Theorem~\ref{thm:cy3 K=0}. Then $X$, being a finite quotient of $X_1$, is also projective.
\end{rem}

\subsection{Deformation theory} \label{sec:deformation theory}

We collect some notation and basic facts from deformation theory.

\begin{dfn}[Deformations of complex spaces]
Let $X$ be a compact complex space. A \emph{deformation} of $X$ is a proper flat morphism $\mf X \to (B, 0)$ from a (not necessarily reduced) complex space $\mf X$ to a complex space germ $(B, 0)$, together with the choice of an isomorphism $\mf X_0 \isom X$.
We usually suppress both the base point $0 \in B$ and the choice of isomorphism from notation.
A deformation over $B = S_1 := \Spec \C[\eps]/(\eps^2)$ is called a \emph{first-order (infinitesimal) deformation}.
\end{dfn}

\begin{dfn}[Locally trivial deformations]
A deformation $\pi\!: \mf X \to B$ is called \emph{locally trivial} if for every $x \in \mf X_0$ there exist open subsets $0 \in B^\circ \subset B$ and $x \in U \subset \pi^{-1}(B^\circ)$ and an isomorphism
\[ \xymatrix{
U \ar^-\sim[rr] \ar_-\pi[dr] & & (\mf X_0 \cap U) \x B^\circ \ar^-{\pr_2}[dl] \\
& B^\circ. &
} \]
\end{dfn}

This definition is in obvious conflict with the usual notion of local triviality (i.e.~local on the base). Nevertheless, the terminology is standard (see e.g.~\cite{FK87, Ser06}) and we trust that it will not lead to any confusion.

\begin{dfn}[Versal and miniversal deformations]
A deformation $\mf X \to B$ of $X$ is called \emph{versal} if every deformation $\mf X' \to B'$ of $X$ can be obtained from it by a base change $\mu\!: B' \to B$.
It is called \emph{miniversal} (or \emph{semiuniversal}) if additionally the differential $\d_0 \mu$ is uniquely determined.
Ditto for locally trivial deformations.
\end{dfn}

We denote by $\Def(X)$ the miniversal deformation space and by $\Deflt(X)$ the miniversal locally trivial deformation space of $X$.
These exist by~\cite[Hauptsatz, p.~140]{Gra74} and~\cite[Cor.~0.3]{FK87}, respectively. Furthermore $\Deflt(X)$ is a subgerm of $\Def(X)$ by the proof of~\cite[Cor.~0.3]{FK87}.

To a first-order deformation $\pi\!: \mf X \to S_1$ one can associate its \emph{Kodaira--Spencer class} in $\Ext^1(\Omega_X^1, \O_X)$, which is the extension class of the conormal sequence
\[ 0 \lto \O_X \lto \Omega_{\mf X}^1\big|_{X_0} \stackrel{\pi^*}{\lto} \Omega_X^1 \lto 0. \]
First-order deformations are classified by their Kodaira--Spencer class, i.e.~we have $T_0 \Def(X) = \Ext^1(\Omega_X^1, \O_X)$~\cite[Thm.~1.1.10]{Ser06}.
Under this correspondence, locally trivial deformations correspond to locally split extensions of $\Omega_X^1$ by $\O_X$, that is, $T_0 \Deflt(X) = H^1(X, \sT_X)$~\cite[Prop.~1.2.9]{Ser06}. The inclusion $T_0 \Deflt(X) \subset T_0 \Def(X)$ is given by the Grothendieck spectral sequence
\[ E_2^{p,q} = H^p \big( X, \sExt^q(\Omega_X^1, \O_X) \big) \Rightarrow \Ext^{p+q}(\Omega_X^1, \O_X). \]
The \v Cech cohomological description of the Kodaira--Spencer class of a locally trivial first-order deformation works the same way as it does for deformations of complex manifolds~\cite[Sec.~9.1.2]{Voi02}.
For a locally trivial deformation $\pi\!: \mf X \to B$ over an arbitrary base $B$, we obtain correspondingly a \emph{Kodaira--Spencer map} $\kappa(\pi)\!: T_0 B \to H^1(X, \sT_X)$.

\begin{dfn}[Algebraic approximations] \label{dfn:alg approx}
Let $X$ be a compact complex space and $\pi\!: \mf X \to B$ a deformation of $X$.
Consider the set of projective fibres
\[ B^{\mathrm{alg}} := B_\pi^{\mathrm{alg}} := \big\{ t \in B \;\big|\; \mf X_t \text{ is projective} \big\} \subset B \]
and its closure $\overline{B^{\mathrm{alg}}} \subset B$.
\begin{enumerate}
\item We say that $\mf X \to B$ is an \emph{algebraic approximation of $X$} if $0 \in \overline{B^{\mathrm{alg}}}$.
\item We say that $\mf X \to B$ is a \emph{strong algebraic approximation of $X$} if $\overline{B^{\mathrm{alg}}}$ contains an open neighborhood of $0 \in B$.
\end{enumerate}
\end{dfn}

\PreprintAndPublication{
\section{The Albanese map of a threefold with \texorpdfstring{$\kappa = 0$}{kappa = 0}} \label{sec:alb kappa=0}

The following result was proved by Ueno~\cite[Main Thm.~I]{Uen87} for all possible values of $q(X)$ except $q(X) = 2$. Hence in the proof we will concentrate on this particular case.

\begin{thm} \label{thm:alb bim bundle}
Let $X$ be a compact complex manifold of Fujiki class $\cC$ with $\dim X = 3$ and $\kappa(X) = 0$.
Then the Albanese map $\alpha\!: X \to A := \Alb(X)$ is bimeromorphic to a surjective analytic fibre bundle with connected fibre, i.e.~there is a diagram
\[ \xymatrix{
X \ar_-\alpha[dr] \ar@{-->}^-\pi[rr] & & Y \ar^-f[dl] \\
& A &
} \]
where $Y$ is a compact complex manifold, $\pi$ is bimeromorphic and $f$ is an analytic fibre bundle with $f_* \O_Y = \O_A$.
\end{thm}

\subsection{A result of Fujiki}

In the proof of Theorem~\ref{thm:alb bim bundle}, we will need a proposition of Fujiki~\cite[Prop.~4.5]{Fuj83} about the irregularity of an isotrivial fibre space.
As the availability of~\cite{Fuj83} is somewhat limited, we have chosen to reproduce the statement in question here for the reader's convenience.
By $D / \Gamma$ we denote the period domain of weight $1$ integral Hodge structures.

\begin{prp}[Additivity of irregularity] \label{prp:Fuj83}
Let $f\!: X \to Y$ be a fibre space of compact complex manifolds in $\cC$. Let $U \subset Y$ be a Zariski open subset over which $f$ is smooth. Suppose that the period map $\Phi\!: U \to D/\Gamma$ associated to $f_U$ is constant. Then there exist a normal compact complex space $\wt Y$ and a finite covering $\wt Y \to Y$ which is \'etale over $U$ such that if $\wt X := X \x_Y \wt Y$, then $q(\wt X) = q(X_y) + q(\wt Y)$ for any $y \in U$.
\end{prp}

\begin{proof}
See~\cite[Prop.~4.5]{Fuj83}.
\end{proof}

\subsection{The case of irregularity two}

As already remarked, we are mainly interested in the situation where $q(X) = 2$.
The following proposition says that in this case Theorem~\ref{thm:alb bim bundle} is true after a finite \'etale base change.

\begin{prp} \label{prp:alb bim bundle}
Let $X$ be a compact complex manifold in class $\cC$ with $\dim X = 3$, $\kappa(X) = 0$, and $q(X) = 2$. Let $\alpha\!: X \to A := \Alb(X)$ be the Albanese map of $X$. Then there is a diagram
\begin{sequation} \label{eqn:prp alb bim bundle}
\xymatrix{
  Y \ar_-f[drr] & & \wt X \ar@{-->}_-{\pi, \textit{ bimerom.}}[ll] \ar[rr] \ar[d] & & X \ar^-{\alpha}[d] \\
  & & \wt A \ar^-{\beta, \textit{ finite \'etale}}[rr] & & A
}
\end{sequation}%
where $Y$ is a complex torus, $\wt X = X \x_A \wt A$, and $f$ is an analytic fibre bundle with $f_* \O_Y = \O_{\wt A}$.
\end{prp}

\begin{proof}
By~\cite[Main Thm.~I]{Uen87}, the map $\alpha$ is surjective and has connected fibres.
Let $F \subset X$ be a smooth fibre of $\alpha$. By the easy addition formula, we have $\kappa(F) \ge 0$.
Then by the solution of the Iitaka--Viehweg $C_{3,2}^+$ conjecture~\cite[Thms.~2.1 and 2.2]{Uen87}, it follows that $\kappa(F) = 0$ and the family given by $\alpha$ has variation zero. This means that all smooth fibres of $\alpha$ are isomorphic to some fixed elliptic curve $E$. 

\begin{clm} \label{clm:X0}
There exists a compact complex manifold $X_0$ bimeromorphic to $X$ such that the Albanese map $X_0 \to \Alb(X_0) = A$ is smooth outside a codimension two (i.e.~finite) subset of $A$.
\end{clm}

\begin{proof}[Proof of Claim~\ref{clm:X0}]
First note that for any bimeromorphic map $\phi\!: X \dashrightarrow X_0$ the composition $\alpha \circ \phi^{-1}\!: X_0 \dashrightarrow X \to A$ is holomorphic, as $A$ does not contain any rational curves. In fact, $A = \Alb(X_0)$ and $\alpha \circ \phi^{-1}$ is the Albanese map of $X_0$.

By~\cite[Cor.~1.10]{Uen87}, there is a bimeromorphic map $\phi\!: X \dashrightarrow X_0$ such that $\alpha_0 = \alpha \circ \phi^{-1}$ has the following property:
Let $U \subset A$ be the largest open set such that $\alpha_0^{-1}(U) \to U$ is smooth, and let $D \subset A$ be a divisorial component of $A \minus U$. If $H \subset A$ is a non-singular analytic arc (i.e.~the image of the unit disc under a holomorphic embedding) that intersects $D$ transversely in a single smooth point and $\alpha_0^{-1}(H)$ is smooth, then the elliptic surface $\alpha_0^{-1}(H) \to H$ does not contain any $(-1)$-curves.

In this situation we know by~\cite[Thm.~2.4]{Uen87} that if $A \minus U$ actually contains an irreducible divisor $D$, then
\begin{sequation} \label{eqn:can bdl formula}
\kappa(X_0) \ge \kappa \big( A, \O_A(D) \big).
\end{sequation}%
By the Theorem of the Square~\cite[Thm.~2.3.3]{BL04}, for any $a \in A$ we have
\[ \tau_a^* D + \tau_{-a}^* D \in |2D|, \]
where $\tau_a\!: A \to A$ denotes translation by $a$.
Clearly $\tau_a^* D \ne D$ for general $a \in A$, thus $h^0 \big( A, \O_A(2D) \big) \ge 2$ and $\kappa \big( A, \O_A(D) \big) \ge 1$.
Together with~\eqref{eqn:can bdl formula}, this contradicts the fact that $\kappa(X_0) = 0$. Claim~\ref{clm:X0} is proven.
\end{proof}

The manifold $X_0$ is again in class $\cC$, and for any finite \'etale base change $\wt A \to A$ the fibre product $X_0 \x_A \wt A$ will be bimeromorphic to $X \x_A \wt A$.
Hence we may replace $X$ by $X_0$ and make the following additional assumption without loss of generality.

\begin{awlog} \label{awlog:alpha smooth}
The Albanese map $\alpha\!: X \to A$ is smooth outside a codimension two subset of $A$.
\end{awlog}

As the family defined by $\alpha$ has variation zero, by Proposition~\ref{prp:Fuj83} there is a finite cover $\beta\!: \wt A \to A$, \'etale over the smooth locus of $\alpha$, such that for $\wt X := X \x_A \wt A$ we have $q(\wt X) = q(\wt A) + q(E)$.
The cover $\beta$ is \'etale by Assumption~\ref{awlog:alpha smooth} and purity of branch locus~\cite[Theorem]{Nag59}.
Hence $\wt A$ is again a two-dimensional complex torus and we get $q(\wt X) = 3$.
As the manifold $\wt X$ still belongs to the class $\cC$, by~\cite[Main Thm.~I]{Uen87} the Albanese map $\pi\!: \wt X \to Y := \Alb(\wt X)$ is bimeromorphic. By the universal property of the Albanese torus (Definition~\ref{dfn:alb}), $\wt X \to \wt A$ factors through $Y$, yielding the following diagram.
\[ \xymatrix{
Y \ar_-f[drr] & & \wt X \ar_-\pi[ll] \ar[rr] \ar[d] & & X \ar^-{\alpha}[d] \\
& & \wt A \ar^-\beta[rr] & & A
} \]
The map $f$, being a surjective morphism of complex tori, clearly is a fibre bundle. Furthermore $f$ has connected fibres since $\alpha$ does. This ends the proof.
\end{proof}

\subsection{Proof of Theorem~\ref{thm:alb bim bundle}}

By~\cite[Main Thm.~I]{Uen87}, the map $\alpha$ is surjective with connected fibres. Hence $0 \le q(X) \le 3$.
If $q(X) \in \{ 0, 1, 3 \}$, then the result holds by~\cite[Main Thm.~I]{Uen87}.
Hence we will assume from now on that $q(X) = 2$.

Consider the diagram~\eqref{eqn:prp alb bim bundle} given by Proposition~\ref{prp:alb bim bundle}.
After fixing base points, the \'etale cover $\beta\!: \wt A \to A$ corresponds to a transitive $\pi_1(A)$-action on some finite set $S$, i.e.~a group homomorphism $\rho\!: \pi_1(A) \to \Aut(S)$. The subgroup $\pi_1(\wt A) \subset \pi_1(A)$ is the stabilizer of the basepoint of $\wt A$. The kernel $\ker(\rho)$ is a finite index \emph{normal} subgroup of $\pi_1(A)$ contained in $\pi_1(\wt A)$. It corresponds to a finite Galois cover $\wh A \to A$ that factors through $\beta$.
Replacing $\wt A$ by $\wh A$, we may assume that $\beta$ is Galois. (Actually, the covering constructed by Fujiki in the proof of Proposition~\ref{prp:Fuj83} is already Galois.)

\begin{awlog} \label{awlog:Galois}
There is a finite group $G$ acting fixed point freely on $\wt A$ such that $A = \wt A / G$ and $\beta\!: \wt A \to A$ is the quotient map.
\end{awlog}

We will consider $G$ as a subgroup of $\Aut(\wt A)$.
Every element $\sigma \in G$ induces an automorphism $\wt\sigma := \id_X \x_A \sigma$ of $\wt X$. Hence $G$ acts on $\wt X$ such that $\wt X \to \wt A$ is equivariant and $X = \wt X / G$.
Every $\sigma \in G$ also induces a bimeromorphic automorphism $\sigma_Y := \pi \circ \wt\sigma \circ \pi^{-1}$ of $Y$.
But $Y$, being a complex torus, does not contain any rational curves.
Hence $\sigma_Y \in \Aut(Y)$, i.e.~$\sigma_Y$ is in fact biholomorphic.

This implies that $G$ acts on $Y$ in such a way that $f$ is equivariant. As the $G$-action on $\wt A$ is fixed point free, so is the $G$-action on $Y$. Hence $Y / G \to A = \wt A / G$ is again a fibre bundle, and $Y / G$ is bimeromorphic over $A$ to $\wt X / G = X$.
Theorem~\ref{thm:alb bim bundle} is proved. \qed
}{}

\section{The Albanese map for singular spaces}

It is a folklore lemma that ``the Albanese map exists for projective varieties with rational singularities''~\cite[Prop.~2.3]{Rei83}, \cite[Lemma~0.3.3]{Som86}.
Projectivity is of course not necessary, and also the assumption on the singularities can be weakened.
In Theorem~\ref{thm:alb ratl} below we try to identify the largest class of compact complex spaces $X$ for which the Albanese map, defined by a universal property for maps to complex tori, exists and is related to the Albanese map of a resolution of $X$.

\begin{dfn}[Albanese map] \label{dfn:alb}
Let $X$ be a compact complex space. A holomorphic map $\alpha\!: X \to A$ to a complex torus $A$ is said to be the Albanese map of $X$, and $A$ is called the Albanese torus of $X$, if any map from $X$ to a complex torus $T$ factors uniquely through $\alpha$.
\[ \xymatrix{
X \ar[rr] \ar_-\alpha[dr] & & T \\
& A \ar@{..>}_-{\exists!}[ur] &
} \]
\end{dfn}

The Albanese torus is unique up to isomorphism if it exists. We write $A = \Alb(X)$ and $\alpha = \alb_X$.

\begin{exm} \label{exm:alb}
Let $X$ be the nodal curve obtained from an elliptic curve $E$ by identifying two points $p \ne q \in E$.
Then $h^1(X, \O_X) = 2$. However, if $\O_E(p - q) \in \Pic^0(E)$ is non-torsion, then $\Alb(X) = 0$.
\end{exm}

As Example~\ref{exm:alb} shows, even if it exists $\Alb(X)$ need not be related to $h^1(X, \O_X)$ or to $\Alb(\wt X)$, where $\wt X \to X$ is a resolution of singularities.
The situation is much better for spaces with rational singularities.

\begin{thm}[Albanese map for spaces with rational singularites] \label{thm:alb ratl}
Let $X$ be a normal compact complex space, and let $f\!: \wt X \to X$ be a projective resolution of singularities. Assume either of the following two conditions.
\begin{enumerate}
\item\label{itm:alb ratl.1} The higher direct image sheaf $R^1 f_* \Z_{\wt X}$ vanishes.
\item\label{itm:alb ratl.2} The natural map $H^1 \big( \wt X, \O_{\wt X} \big) \to H^0 \big( X, R^1 f_* \O_{\wt X} \big)$ is zero.
\end{enumerate}
Then $X$ admits an Albanese torus, $\Alb(X) = \Alb(\wt X)$, and the following diagram commutes:
\[ \xymatrix{
\wt X \ar^-{\alb_{\wt X}}[rr] \ar_-f[d] & & \Alb(\wt X) \ar@{=}[d] \\
X \ar^-{\alb_X}[rr] & & \Alb(X).
} \]
\end{thm}

\begin{rem}
It is clear that condition~(\ref{thm:alb ratl}.\ref{itm:alb ratl.2}) is satisfied for a space $X$ with rational singularities.
The pushforward of the exponential sequence on $\wt X$ yields an injection $R^1 f_* \Z_{\wt X} \inj R^1 f_* \O_{\wt X}$, hence also~(\ref{thm:alb ratl}.\ref{itm:alb ratl.1}) holds for such $X$.
\end{rem}

\begin{proof}[Proof of Theorem~\ref{thm:alb ratl}]
Let $\wt\alpha\!: \wt X \to \Alb(\wt X)$ be the Albanese map of $\wt X$, see~\cite[p.~163]{Bla56}. By Lemma~\ref{lem:alb ratl} below, the map $\wt\alpha$ factors through a map $\alpha\!: X \to \Alb(\wt X)$. By Lemma~\ref{lem:alb ratl} again, $\alpha$ has the desired universal property.
\end{proof}

\begin{lem} \label{lem:alb ratl}
Let $X$ be a normal complex space (not necessarily compact) and $f\!: \wt X \to X$ a projective resolution such that~(\ref{thm:alb ratl}.\ref{itm:alb ratl.1}) or~(\ref{thm:alb ratl}.\ref{itm:alb ratl.2}) is satisfied.
Then any map $\phi\!: \wt X \to T$ to a complex torus $T$ factors through $f$.
\end{lem}

\begin{proof}
As $X$ is assumed to be normal, it is sufficient to prove the following:
For any $p \in X$, the fibre $W := f^{-1}(p)$ is contracted to a point by $\phi$.

In case~(\ref{thm:alb ratl}.\ref{itm:alb ratl.1}), the assumption implies $H^1(W, \Z) = 0$. Thanks to the universal coefficient theorem and the fact that $W$ is connected, $H^1(W, \Z) = \Hom(\pi_1(W), \Z)$.
As $\pi_1(T)$ is free abelian, the map $\pi_1(W) \to \pi_1(T)$ induced by $\phi|_W$ is trivial and $\phi|_W$ lifts to the universal cover $\wt T$ of $T$.
Since $\wt T$ is Stein, the lift is constant, hence so is $\phi|_W$ itself.

In case~(\ref{thm:alb ratl}.\ref{itm:alb ratl.2}), assume that $\phi(W)$ is positive-dimensional.
Blowing up $\wt X$ further we may assume that $W$ is a simple normal crossings divisor.
Let $W_1 \subset W$ be an irreducible component such that $\phi(W_1)$ is positive-dimensional. Then there exists $\sigma \in H^1(T, \O_T)$ such that $(\phi|_{W_1})^* \sigma \ne 0 \in H^1(W_1, \O_{W_1})$.
It follows that for any neighborhood $p \in U \subset X$, setting $\wt U = f^{-1}(U)$ we have $(\phi^* \sigma)\big|_{\wt U} \ne 0 \in H^1(\wt U, \O_{\wt U})$.
Hence the image of $\phi^* \sigma \in H^1(\wt X, \O_{\wt X})$ in $H^0(X, R^1 f_* \O_{\wt X})$ is nonzero. This contradicts the assumption~(\ref{thm:alb ratl}.\ref{itm:alb ratl.2}).
\end{proof}

\section{Decomposing the Albanese map} \label{sec:lt alb}

In this section, we first prove Theorem~\ref{thm:alb cy3 intro} from the introduction.
This will then yield the following statement about Calabi--Yau threefolds with trivial canonical sheaf.

\begin{thm} \label{thm:cy3 K=0}
Let $X$ be a \cy threefold (in the sense of Definition~\ref{dfn:cy}) with only isolated canonical singularities and $\omega_X \isom \O_X$.
If $X$ is non-projective, then it is smooth.
\end{thm}

\subsection{Local group actions and vector fields}

For the reader's convenience, we recall the definition of a local group action.
Our presentation follows~\cite[Sec.~4.A]{GKK10}, but see also~\cite[Sec.~4]{Kau65}.

\begin{dfn}[Local group action] \label{dfn:local grp action}
Let $G$ be a connected complex Lie group and $Z$ a reduced complex space. A \emph{local $G$-action} is given by a holomorphic map $\Phi\!: \Theta \to Z$, where $\Theta \subset G \x Z$ is an open neighborhood of $\{ e \} \x Z$, such that:
\begin{enumerate}
\item For all $z \in Z$, the subset $\{ g \in G \;|\; (g, z) \in \Theta \}$ is connected.
\item We have $\Phi(e, z) = z$ for all $z \in Z$.
\item If $(gh, z), (h, z), (g, \Phi(h, z)) \in \Theta$, then $\Phi(gh, z) = \Phi(g, \Phi(h, z))$.
\end{enumerate}
As usual, we write $gz$ as a shorthand for $\Phi(g, z)$ if the action is understood.
\end{dfn}

There is a natural notion of \emph{equivalence} of local $G$-actions on $Z$ given by shrinking $\Theta$ to a smaller neighborhood of $\{e\} \x Z$ in $G\times Z$.
An equivalence class of actions induces a Lie algebra homomorphism
\[ \lambda\!: \mf g \to H^0 \big(Z, \sT_Z \big), \]
where $\mf g$ is the Lie algebra of $G$. In terms of derivations $\O_Z \to \O_Z$, the vector field $\lambda(\xi)$ sends a function $f \in \O_Z(U)$ to
\[ z \mapsto \frac\d{\d t} \bigg|_{t=0} f \big( \exp(t \xi) z \big). \]
Here $\exp\!: \mf g \to G$ is the exponential map.
The following theorem tells us that there is an inverse to this construction.

\begin{thm}[Vector fields and local group actions] \label{thm:vf grp action}
If $\lambda\!: \mf g \to H^0 \big(Z, \sT_Z \big)$ is a homomorphism of Lie algebras, then up to equivalence (i.e.~shrinking $\Theta$ as described above) there exists a unique local $G$-action on $Z$ that induces the given $\lambda$.
\end{thm}

\begin{proof}
See~\cite[Satz 3]{Kau65}.
\end{proof}

\subsection{Proof of Theorem~\ref{thm:alb cy3 intro}}

As the proof is somewhat long, we divide it into several steps.

\subsubsection*{Step 0: Setup of notation}

We may assume that $q := q(X) > 0$, otherwise the statement is empty.
Let $\pi\!: \wt X \to X$ be a projective resolution of singularities. Then $\wt X$ is a compact \kahler manifold and since $X$ has canonical singularities, $\wt X$ has Kodaira dimension $\kappa(\wt X) = \kappa(X, K_X) = 0$ by Proposition~\ref{prp:cy3 abundance}.
Since canonical singularities are rational, by Theorem~\ref{thm:alb ratl} the Albanese map $\wt\alpha$ of $\wt X$ is given by
\[ \wt\alpha\!: \wt X \xrightarrow{\quad\pi\quad} X \xrightarrow{\quad\alpha\quad} A = \Alb(\wt X). \]
By \PreprintAndPublication{Theorem~\ref{thm:alb bim bundle}}{\cite[Main Thm.~I]{Uen87}}, the map $\wt\alpha$ is bimeromorphic to an analytic fibre bundle $f\!: Y \to A$ with smooth connected fibre $F$.
We necessarily have $\kappa(F) = 0$.
The situation can be summarized in the following diagram.
\begin{sequation} \label{eqn:diag}
\xymatrix{
X \ar_-{\alb_X = \alpha}[drr] \ar@/^1.75pc/@{-->}^-{\text{$\beta$, bimerom.}}[rrrr] & & \wt X \ar_-{\quad\text{$\pi$, resol.}}[ll] \ar@{-->}^-{\text{bimerom.}}[rr] \ar^-{\wt\alpha}[d] & & Y \supset F \ar^-{\quad\text{$f$, fibre bundle}}[dll] \\
& & A & &
}
\end{sequation}%

\subsubsection*{Step 1: The curvature of $Y$}

If $q = 1$, then the fibre $F$ is two-dimensional.
Let $F \to F'$ be the minimal model of $F$.
We run a Minimal Model Program relative to the map $f$. Since $f$ is a fibre bundle, every step in this MMP is the inverse of the blowup of an \'etale multisection of $f$. Hence the end product of this MMP is a fibre bundle $Y' \to A$ with fibre $F'$.
Replacing $Y$ by $Y'$, we may assume the following.

\begin{awlog} \label{awlog:F min}
If $q = 1$, then $F$ is a minimal surface.
\end{awlog}

This has the following important consequence.

\begin{clm} \label{clm:numtriv}
The manifold $Y$ carries a Ricci-flat \kahler metric.
\end{clm}

\begin{proof}[Proof of Claim~\ref{clm:numtriv}]
By Assumption~\ref{awlog:F min} and abundance for surfaces, we may pick $m > 0$ such that $\omega_F^{\tensor m}$ is trivial.
Possibly replacing $m$ by a multiple, there is a nonzero section $s \in H^0(Y, mK_Y)$.
For every fibre $Y_a$, the restriction $s|_{Y_a}$ is either zero or non-vanishing everywhere.
If $s|_{Y_a}$ is zero for some $a \in A$, then there is a nonzero effective divisor $a \in D \subset A$ such that $f^* D \le mK_Y$.
But then $0 = \kappa(Y) \ge \kappa \big( A, \O_A(D) \big) \ge 1$ by~\cite[Thm.~2.3.3]{BL04}, which is impossible.
Hence $s$ is non-vanishing everywhere, and $mK_Y \sim 0$.
In particular, the first Chern class $c_1(Y)$ vanishes.
On the other hand, $Y$ is \kahler by Lemma~\ref{class C bundle} below.
We conclude by Yau's solution to the Calabi conjecture~\cite{Yau78}.
\end{proof}

\subsubsection*{Step 2: The bimeromorphic geometry of $\beta$}

Let $X^\circ$ be the largest open subset of $X$ such that $\beta$ induces an isomorphism of $X^\circ$ and a dense open subset of $Y$.

\begin{clm} \label{clm:beta}
The map $\beta^{-1}$ in diagram~\eqref{eqn:diag} is a bimeromorphic contraction, i.e.~$\beta$ does not contract any divisors.
Thus $\codim_X(X \minus X^\circ) \ge 2$.
\end{clm}

\begin{proof}[Proof of Claim~\ref{clm:beta}]
Pick a resolution of indeterminacy of $\beta$, i.e.~a diagram
\[ \xymatrix{
& Z \ar_r[dl] \ar^s[dr] & \\
X \ar@{-->}^\beta[rr] & & Y \\
} \]
with $Z$ normal and $r, s$ proper and bimeromorphic.
It suffices to show that if $E_0 \subset Z$ is a prime divisor which is $s$-exceptional, then it is also $r$-exceptional.
To this end, we have the ramification formula
\begin{sequation} \label{eqn:ram}
K_Z \sim_\Q s^* K_Y + E,
\end{sequation}%
where $E$ is $s$-exceptional. Since $Y$ is smooth and hence terminal, $E \ge 0$ is effective and $E_0 \subset \supp(E)$.
Pushing~\eqref{eqn:ram} forward to $X$ yields
\[ K_X \sim_\Q r_* s^* K_Y + r_* E \sim_\Q 0. \]
By Claim~\ref{clm:numtriv}, we get $r_* E \sim_\Q 0$.
As $r_* E$ is effective, this implies $r_* E = 0$.
Then also $r_* E_0 = 0$, as desired.
\end{proof}

\subsubsection*{Step 3: $Y$ splits off a torus}

As an intermediate step, we show that the conclusion we want to prove about $X$ holds for $Y$.

\begin{clm} \label{clm:split Y}
There exists a finite \'etale cover $\phi\!: A_1 \to A$ such that $Y_1 := Y \x_A A_1$ is isomorphic to $F \x A_1$ over $A_1$.
\end{clm}

\begin{proof}[Proof of Claim~\ref{clm:split Y}]
Note that $f\!: Y \to A$ is the Albanese map of $Y$.
Thus, by the universal property in Definition~\ref{dfn:alb}, for every $g \in \Aut(Y)$ there exists a unique $\phi_g\!: A \to A$ such that $\phi_g \circ f = f \circ g$.
If $g \in \Autn(Y)$, the identity component, then $g$ acts trivially on $H^1(Y, \C)$ and in particular on $H^0(Y, \Omega_Y^1)^*$.
Consequently, the linear part of $\phi_g$ is the identity.
That is, $\phi_g$ is a translation by some element $\phi(g) \in A$.
We have thus defined a homomorphism of complex Lie groups $\phi\!: \Autn(Y) \to A$, which we will now show to be the desired cover $A_1 \to A$.
First of all, $Y$ is not uniruled, hence $\Autn(Y)$ is a complex torus by~\cite[Prop.~5.10]{Fuj78}.
It is thus sufficient to show that the induced Lie algebra map $\d\phi\!: H^0(Y, \sT_Y) \to H^0(Y, \Omega_Y^1)^*$ is an isomorphism.
One checks easily that $\d\phi$ is given by the natural contraction pairing
\begin{sequation} \label{contraction}
H^0(Y, \sT_Y) \x H^0(Y, \Omega_Y^1) \lto H^0(Y, \O_Y) = \C.
\end{sequation}%
Let $0 \ne \vec v \in H^0(Y, \sT_Y)$ be a nonzero holomorphic vector field.
Due to the Bochner principle~\cite[Lemma~14.17]{Besse87}, $\vec v$ is parallel with respect to the Chern connection $\nabla$ of the Ricci-flat metric from Claim~\ref{clm:numtriv}.
Dualizing using that metric, $\vec v$ gives rise to a parallel (hence holomorphic, as $\nabla^{0,1} = \bar\del$) $1$-form $\alpha$ on $Y$.
Clearly the contraction $\iota_{\vec v} \alpha \ne 0$.
The argument can also be read backwards, hence \eqref{contraction} is a perfect pairing and $\d\phi$ is an isomorphism.

It remains to show that the family $f_1\!: Y \x_A A_1 \to A_1$ is trivial.
To this end, let $g \in A_1 = \Autn(Y)$ be any point.
Since $\phi_g \circ f = f \circ g$, the automorphism $g$ restricts to an isomorphism $f^{-1}(0) \bij f^{-1}(\phi_g(0)) = f^{-1}(\phi(g))$.
This is the same thing as an isomorphism $f_1^{-1}(0) \bij f_1^{-1}(g)$, independent of any choices.
We have constructed a section of $\Isom_{A_1}(Y \x_A A_1, F \x A_1) \to A_1$.
The claim follows.
\end{proof}

\subsubsection*{Step 4: Vector fields and their flows}

To unclutter notation, we will pull back the whole diagram~\eqref{eqn:diag} to $A_1$ without putting a subscript ``$1$'' everywhere.
In particular, $A$ and $Y$ from now on refer to $A_1$ and $Y_1$, respectively.

Pick a basis $\vec\theta_1, \dots, \vec\theta_q$ of $H^0(A, \sT_A)$, and note that the $\vec\theta_i$ necessarily commute.
Consider the vector fields $\vec w_i := (0, \vec\theta_i)$ on $Y \isom F \x A$.
They induce vector fields $\vec v_i := \beta^{-1}_* (\vec w_i)$ holomorphic on $X^\circ$.
By Claim~\ref{clm:beta} and the reflexivity of the tangent sheaf $\sT_X$, the $\vec v_i$ extend to holomorphic vector fields on all of $X$.
Let $G$ be the complex Lie group $(\C^q, +)$. Its Lie algebra is the abelian Lie algebra $\mf g = \C^q$.
Since $[\vec v_i, \vec v_j] = 0$ for all $1 \le i, j \le q$, the linear map $\lambda\!: \mf g \to H^0 \big( X, \sT_X \big)$ defined by sending the standard basis vectors $e_i \in \mf g$ to $\vec v_i$ is a homomorphism of Lie algebras.
By Theorem~\ref{thm:vf grp action}, the map $\lambda$ is induced by a unique local $\C^q$-action on $X$, given by a holomorphic map
\[ \C^q \x X \supseteq \Theta \xrightarrow{\Phi} X \]
as in Definition~\ref{dfn:local grp action}.
Since $X$ is compact, this action is defined everywhere, i.e.~$\Theta = \C^q \x X$.
Similarly, the vector fields $\vec\theta_i$ on $A$ give rise to a (global) $\C^q$-action on $A$,
\[ \C^q \x A \xrightarrow{\Psi} A, \]
which induces the action $\psi$ of $A$ on itself by translation.
Since $\alpha_* \vec v_i = \vec\theta_i$ in the sense of Definition~\ref{dfn:projectable} below, the diagram
\[ \xymatrix{
\C^q \x X \ar_-{\id \x \alpha}[d] \ar^-\Phi[r] & X \ar^-{\alpha}[d] \\
\C^q \x A \ar^-\Psi[r] & A
} \]
commutes.
We obtain an induced action $\phi\!: A \x X \to X$ that commutes with $\psi$.
Its restriction to $A \x X_a \to X$, for any fibre $X_a$, is an isomorphism.
This completes the proof of Theorem~\ref{thm:alb cy3 intro}. \qed

\begin{lem}[Class $\cC$ fibre bundles] \label{class C bundle}
Let $X, B$ be compact complex manifolds and $f\!: X \to B$ an analytic fibre bundle, with fibre $F$.
Assume that $F$ and $B$ are \kahler and that $X$ is of class $\cC$.
If $\dim X \le 3$, then $X$ is \kahler.
\end{lem}

\begin{proof}
We may assume that $\dim B \in \{ 1, 2 \}$.
Let $\mu\!: \wt X \to X$ be a proper bimeromorphic map from a \kahler manifold $(\wt X, \wt\omega)$.
Then $\omega := \mu_* \wt\omega$ is a \kahler current on $X$.
Since $\dim F \le 2$, for general $b \in B$ the restriction of $[\omega] \in H^2(X, \R)$ to the fibre $X_b$ is a \kahler class by~\cite{Miy74}.
As $f$ is locally trivial, this is even true for all $b \in B$.
Now pick any \kahler class $[\omega_B]$ on $B$.
It follows that for $\lambda \gg 0$ sufficiently large, $[\omega] + \lambda \cdot f^* [\omega_B]$ is \kahler on $X$.
\end{proof}

\subsection{Threefolds with trivial canonical bundle}

Theorem~\ref{thm:cy3 K=0} is a consequence of the following two propositions.

\begin{prp} \label{prp:q le 1}
Let $X$ be a \cy threefold. If $X$ is not smooth, then its singular locus $X_\sg$ satisfies
\[ q(X) \le \dim X_\sg \le 1. \]
\end{prp}

\begin{proof}
By Theorem~\ref{thm:alb cy3 intro}, the Albanese map $\alpha\!: X \to A = \Alb(X)$ of $X$ is an analytic fibre bundle, say with fibre $F$.
If $U \subset A$ is an open subset such that $X_U \isom U \x F$, then $(X_U)_\sg \isom U \x F_\sg$.
Since $X$ is assumed to be non-smooth, we necessarily have $F_\sg \ne \emptyset$.
It follows that $\dim X_\sg \ge \dim U = q(X)$.
Furthermore it is clear that $\dim X_\sg \le 1$, as $\dim X = 3$ and $X$ is normal.
\end{proof}

\begin{prp}[Kodaira embedding theorem] \label{prp:Kodaira}
Let $X$ be a normal compact \kahler space with rational singularities. If $H^2(X, \O_X) = 0$, then $X$ is projective.
\end{prp}

\begin{proof}
Let $f\!: \wt X \to X$ be a projective resolution of singularities.
Then $\wt X$ is a compact \kahler manifold.
Since $X$ has rational singularities, we have $H^2(\wt X, \O_{\wt X}) = H^2(X, \O_X) = 0$ by the Leray spectral sequence associated to $f$ and $\O_{\wt X}$.
Thus $H^{1,1}(\wt X) = H^2(\wt X, \C)$.
The \kahler cone of $\wt X$ is open in $H^{1,1}(\wt X) \cap H^2(\wt X, \R) = H^2(\wt X, \R)$, hence there exists an integral \kahler class on $\wt X$.
By the Kodaira embedding theorem~\cite[Thm.~7.11]{Voi02}, it follows that $\wt X$ is projective.
In particular, $X$ is Moishezon.
A Moishezon \kahler space with rational singularities is projective by~\cite[Thm.~1.6]{Nam02}.
\end{proof}

\begin{proof}[Proof of Theorem~\ref{thm:cy3 K=0}]
We show that if $X$ is not smooth, then it is projective.
By assumption, $X$ has only isolated singularities, so $\dim X_\sg = 0$.
Proposition~\ref{prp:q le 1} then implies that $q(X) = h^1(X, \O_X) = 0$.
The space $X$ has rational singularities, in particular it is Cohen--Macaulay.
Hence by Serre duality on complex spaces~\cite[Ch.~VII, Thm.~3.10]{BS76} we get $h^2(X, \omega_X) = h^1(X, \O_X) = 0$.
As the canonical sheaf of $X$ is assumed to be trivial, it follows that $H^2(X, \O_X)$ vanishes.
Then $X$ is projective by Proposition~\ref{prp:Kodaira}.
\end{proof}

\section{Finite group actions on sheaves and cohomology} \label{sec:finite grp actions}

In this section, we will consider a finite group acting on a coherent sheaf. The results we give may be well-known to experts. For example, Lemma~\ref{lem:cohom inv} is contained in an unpublished manuscript of Koll\'ar. However, we have been unable to find published proofs of these results.

We fix the following setup: $X$ is a normal complex space and $G$ is a finite group acting on $X$. We denote the quotient morphism by $\pi\!: X \to Y = X/G$.

\begin{dfn}[$G$-sheaves] \label{dfn:G-sheaf}
A \emph{$G$-sheaf} $\sF$ is a coherent sheaf on $X$ together with isomorphisms $\sigma(g)\!: g^* \sF \bij \sF$ satisfying $\sigma(gh) = \sigma(h) \circ h^* \sigma(g)$ for all $g, h \in G$.
If $G$ acts trivially on $X$, then the \emph{invariant subsheaf} $\sF^G \subset \sF$ is defined by
\[ \sF^G(U) := \sF(U)^G \quad \text{for any open subset $U \subset X$.} \]
Here $\sF(U)^G$ denotes the subgroup of $G$-invariant elements of $\sF(U)$.
\end{dfn}

If $\sF$ is a $G$-sheaf, the push-forward sheaf $\pi_* \sF$ is a $G$-sheaf on $Y$, where $G$ acts trivially on $Y$. The associated invariant subsheaf will be denoted by $(\pi_* \sF)^G$.

Note that there is a natural $G$-action on the cohomology groups $H^i(X, \sF)$. The following lemma shows that ``taking cohomology commutes with taking invariants''.

\begin{lem}[Cohomology of invariant subsheaves] \label{lem:cohom inv}
For all $i \ge 0$, there are natural isomorphisms
\[ H^i \big( Y, (\pi_* \sF)^G \big) \isom H^i(X, \sF)^G. \]
\end{lem}

\begin{proof}
We have natural maps
\[ H^i \big( Y, (\pi_* \sF)^G \big) \lto H^i(Y, \pi_* \sF) \isom H^i(X, \sF), \]
the latter isomorphism coming from the fact that $\pi$ is finite.
Averaging over the group $G$ gives us a Reynolds operator $R\!: \pi_* \sF \surj (\pi_* \sF)^G$, that is, a splitting of the inclusion $(\pi_* \sF)^G \subset \pi_* \sF$.
Consequently, $H^i \big( Y, (\pi_* \sF)^G \big) \to H^i(X, \sF)$ is injective, and it clearly factorizes via $\phi\!: H^i \big( Y, (\pi_* \sF)^G \big) \to H^i(X, \sF)^G$. This is the map whose existence is claimed in the lemma, and it remains to show its surjectivity.
But the Reynolds operator $R$ induces a map $H^i(X, \sF) \to H^i \big( Y, (\pi_* \sF)^G \big)$, whose restriction to $H^i(X, \sF)^G$ is a right inverse to $\phi$. This ends the proof.
\end{proof}

Now we will specialize to the case where $\sF = \sT_X$ is the tangent sheaf of $X$, which is a $G$-sheaf in a natural way.

\begin{lem}[Trace map for vector fields] \label{lem:trace}
There is a natural morphism
\[ \tr\!: \pi_* \sT_X \lto \sT_Y, \]
called the trace map. If $G$ acts freely in codimension one, then $\tr$ induces an isomorphism $(\pi_* \sT_X)^G \to \sT_Y$. In particular, in this case we have
\[ H^i(Y, \sT_Y) \isom H^i(X, \sT_X)^G \quad \text{for all $i \ge 0$.} \]
\end{lem}

\begin{proof}
Let $\Der(\sA)$ denote the vector space of $\C$-linear derivations of a sheaf $\sA$ of $\C$-algebras.
On an open set $U \subset Y$, the trace map
\[ \tr(U)\!: \pi_* \sT_X(U) = \Der \big( \O_{\pi^{-1}(U)} \big) \lto \sT_Y(U) = \Der(\O_U) \]
is defined by sending $D \in \Der \big( \O_{\pi^{-1}(U)} \big)$ to
\[ \tr(D) := \tr_{X/Y} \circ D \circ \pi^*\!: \O_U \lto (\pi_* \O_X)\big|_U \lto (\pi_* \O_X)\big|_U \lto \O_U. \]
Here $\tr_{X/Y}$ is the usual trace map for functions~\cite[Def.~5.6]{KM98}. It is easy to check that $\tr(D)$ really is a derivation of $\O_U$. Note that this definition works for any finite surjective morphism of normal complex spaces.
Note also that, returning to vector field notation,
\begin{sequation} \label{eqn:tr}
\tr(\vec v)(y) = \sum_{\pi(x) = y} \d\pi \big( \vec v(x) \big) \in T_y Y
\end{sequation}%
for any smooth point $y \in Y$ over which $\pi$ is \'etale.

Now if the $G$-action on $X$ is free in codimension one, then $\pi$ is quasi-\'etale.
Let $q\!: X^\circ \to Y^\circ$ be the restriction of $\pi$ to the maximal open subset of $Y$ over which $\pi$ is \'etale. The differential of $q$ induces an isomorphism $\d q\!: \sT_{X^\circ} \to q^* \sT_{Y^\circ}$ and consequently there is a pullback map for vector fields
\[ \sT_{Y^\circ} \xrightarrow{\hspace{2.5em}} q_* q^* \sT_{Y^\circ} \xrightarrow{q_*(\d q)^{-1}} q_* \sT_{X^\circ}. \]
Since $\codim_Y (Y \minus Y^\circ) \ge 2$, reflexivity of $\sT_X$ and $\sT_Y$ allows us to extend this map uniquely to a map $\sT_Y \to \pi_* \sT_X$. The pullback of any vector field is clearly $G$-invariant, hence the last map factorizes via
\[ \pi^*\!: \sT_Y \lto (\pi_* \sT_X)^G. \]
Consider now the restriction
\[ \frac{1}{|G|} \tr\!: (\pi_* \sT_X)^G \lto \sT_Y. \]
Using~\eqref{eqn:tr}, it is easy to check that $\pi^*$ and $\frac{1}{|G|} \tr$ are inverse to each other. Hence $\tr$ gives an isomorphism $(\pi_* \sT_X)^G \bij \sT_Y$ as desired. The last statement follows from this and Lemma~\ref{lem:cohom inv}.
\end{proof}

\section{Group actions on miniversal deformations} \label{sec:grp versal}

The following result about finite group actions on the total space of a miniversal deformation and their quotients, split into two separate propositions, is used in the next section. It might be of independent interest.

\begin{prp}[Extending an action to the miniversal deformation] \label{prp:extG}
Let $Y$ be a compact complex manifold, and let $G$ be a finite group acting effectively on $Y$.
Assume that $\Def(Y)$ is smooth, i.e.~that $Y$ has no deformation obstructions.
Then there is a deformation $\pi\!: \mf Y \to \Delta$ of $Y$ over a smooth base $0 \in \Delta$ such that:
\begin{enumerate}
\item\label{itm:extG.1} The group $G$ acts on $\mf Y$ in a fibre-preserving way, extending the given action of $G$ on $\mf Y_0 \isom Y$.
\item\label{itm:extG.2} The Kodaira--Spencer map $\kappa(\pi)\!: T_0 \Delta \to H^1(Y, \sT_Y)$ is an isomorphism onto the $G$-invariant subspace $H^1(Y, \sT_Y)^G \subset H^1(Y, \sT_Y)$.
\end{enumerate}
\end{prp}

\begin{prp}[Quotient of the miniversal deformation] \label{prp:extG2}
In the setting of Proposition~\ref{prp:extG}, put $X := Y/G$ and $\phi\!: \mf X := \mf Y/G \to \Delta$. Then $\phi$ is a locally trivial deformation of $X$.
Its Kodaira--Spencer map $\kappa(\phi)\!: T_0 \Delta \to H^1(X, \sT_X)$ fits into a commutative diagram
\begin{sequation} \label{eqn:KS}
\xymatrix{
T_0 \Delta \ar^-{\kappa(\pi)}[rr] \ar@{=}[d] & & H^1(Y, \sT_Y)^G \ar^{\frac 1 {|G|} \tr}[d] \\
T_0 \Delta \ar^-{\kappa(\phi)}[rr] & & H^1(X, \sT_X).
}
\end{sequation}%
If the $G$-action on $Y$ is free in codimension one, then $\phi\!: \mf X \to \Delta$ is the miniversal locally trivial deformation of $X$. In particular, in this case $\Deflt(X)$ is smooth.
\end{prp}

\subsection{Proof of Proposition~\ref{prp:extG}}

Let $\psi\!: \mf U_Y \to \Def(Y) =: B$ be the miniversal deformation of $Y$. The Kodaira--Spencer map $\kappa(\psi)\!: T_0 B \to H^1(Y, \sT_Y)$ is an isomorphism and consequently $B$ can be embedded as an open neighborhood of $0 \in H^1(Y, \sT_Y)$. The group $G$ acts linearly on $H^1(Y, \sT_Y)$ and after replacing $B$ by $\bigcap_{g \in G} g B$, we may assume that $B$ is $G$-invariant. In particular, $G$ also acts on $B$. Since the $G$-action is linear, the fixed point set $\Delta := B^G$ is smooth with tangent space $T_0 \Delta = H^1(Y, \sT_Y)^G$.

Any finite group is linearly reductive~\cite{Mas99}. Hence by~\cite[Cor.~on p.~225]{Rim80} the group $G$ acts on $\mf U_Y$ in such a way that $\psi$ is equivariant and the restriction to the central fibre is the original $G$-action.
Thus the deformation $\mf Y := \psi^{-1}(\Delta) \to \Delta$ has the required properties. \qed

\subsection{Proof of Proposition~\ref{prp:extG2}}

The proof is divided into four steps. We use the following notation:

\begin{dfn}[Projectable vector fields] \label{dfn:projectable}
Let $f\!: X \to Y$ be a surjective map of normal complex spaces. A vector field $\vec v$ on $X$ is said to be \emph{$f$-projectable} if there exists a (necessarily unique) vector field $\vec w$ on $Y$ such that $\d f(\vec v(x)) = \vec w(f(x))$ for all $x \in X$. In this case we write $f_* \vec v = \vec w$.
\end{dfn}

\subsubsection*{Step 1: $\phi$ is a deformation}
First we check that $\phi$ really is a deformation of $X$.
Denote by $q\!: \mf Y \to \mf X$ the quotient map. Since $\pi$ is proper and $q$ is surjective, $\phi$ is proper. Let $x \in \mf X$ be any point, and pick $y \in q^{-1}(x)$. The local ring $\O_{\mf Y, y}$ is a flat $\O_{\Delta, \pi(y)}$-module, and the trace map $\tr_{\mf Y/\mf X}$ shows that $\O_{\mf X, x}$ is a direct summand of $\O_{\mf Y, y}$. Hence $\O_{\mf X, x}$ is also a flat $\O_{\Delta, \phi(x)}$-module.

\subsubsection*{Step 2: Local triviality of $\phi$}
Let $\{ \cU_\alpha \}_{\alpha \in I}$ be a Stein open covering of $\mf X$, indexed by some set $I$. We put:
\begin{align*}
U_\alpha & = \cU_\alpha \cap X_0. & & \hspace{-5ex} \text{ Then $\{ U_\alpha \}$ is a Stein open covering of $X = X_0$.} \\
\cV_\alpha & = q^{-1}(\cU_\alpha). & & \hspace{-5ex} \text{ Then $\{ \cV_\alpha \}$ is a Stein open covering of $\mf Y$.} \\
V_\alpha & = \cV_\alpha \cap Y_0. & & \hspace{-5ex} \text{ Then $\{ V_\alpha \}$ is a Stein open covering of $Y = Y_0$.}
\end{align*}
Note that $G$ acts on each $\cV_\alpha$ and $V_\alpha$, and that $\cV_\alpha/G = \cU_\alpha$, $V_\alpha/G = U_\alpha$.

We may assume that $\Delta = B_\eps(0) \subset \C^d$ is a polydisc, where $d = \dim \Delta$ and $\C^d$ has coordinates $t_1, \dots, t_d$.
By Lemma~\ref{lem:quotient basis} below, after passing to a refinement of $\{ \cU_\alpha \}$ and consequently also a refinement of $\{ \cV_\alpha \}$, there exist for each $\alpha \in I$ vector fields $\vec v_\alpha^{\,1}, \dots, \vec v_\alpha^{\,d} \in \sT_{\mf Y}(\cV_\alpha)$ which are $\pi$-projectable with
\[ \pi_* \big( \vec v_\alpha^{\,i} \big) = \frac \del {\del t_i} \quad \text{for each $1 \le i \le d$.} \]
Averaging over the group $G$, we may assume that the vector fields $\vec v_\alpha^{\,i}$ are $G$-invariant. Let $\Pi_\alpha^i\!: \cV_\alpha \x \C \to \cV_\alpha$ be the local $\C^+$-action associated to $\vec v_\alpha^{\,i}$ as in Theorem~\ref{thm:vf grp action}.
We obtain an isomorphism $\Pi_\alpha$ as in the following commutative diagram
\begin{sequation} \label{eqn:1044}
\xymatrix{
V_\alpha \x \Delta \ar^-{\Pi_\alpha}_-\sim[rr] \ar_{\pr_2}[dr] & & \cV_\alpha \ar[dl]^\pi \\
& \Delta &
}
\end{sequation}%
by setting
\[ \Pi_\alpha(y, t_1, \dots, t_d) = \Pi_\alpha^d \Big( \Pi_\alpha^{d-1} \big( \cdots \Pi_\alpha^2 \big( \Pi_\alpha^1(y, t_1), t_2 \big), \dots, t_{d-1} \big), t_d \Big). \]
By the uniqueness assertion of Theorem~\ref{thm:vf grp action}, $G$-invariance of $\vec v_\alpha^{\,i}$ implies that we have $\Pi_\alpha^i(gy, t) = g\Pi_\alpha^i(y, t)$ for any $g \in G$, whenever both sides are defined.
This means that $\Pi_\alpha$ becomes $G$-equivariant if we let $G$ act on $V_\alpha \x \Delta$ by the given action $G \looparrowright Y$ on $V_\alpha$ and trivially on $\Delta$.
Hence forming the quotient of diagram~\eqref{eqn:1044} by $G$ yields an isomorphism
\begin{sequation} \label{eqn:1056}
\xymatrix{
U_\alpha \x \Delta \ar^-{\Phi_\alpha}_-\sim[rr] \ar_{\pr_2}[dr] & & \cU_\alpha \ar[dl]^\phi \\
& \Delta. &
}
\end{sequation}%
Since the $\cU_\alpha$ cover $\mf X$, diagram~\eqref{eqn:1056} shows that $\phi\!: \mf X \to \Delta$ is locally trivial.

\subsubsection*{Step 3: The Kodaira--Spencer map}
For the commutativity of~\eqref{eqn:KS}, let $v_0 \in T_0 \Delta$ be an arbitrary tangent vector, and pick a vector field $\vec v$ on $\Delta$ such that $\vec v(0) = v_0$.
For $\alpha \in I$, consider the vector fields
\begin{align*}
\vec v_\alpha := (\Pi_\alpha)_* (0, \vec v) & \in \sT_{\mf Y}(\cV_\alpha), \\
\vec w_\alpha := (\Phi_\alpha)_* (0, \vec v) & \in \sT_{\mf X}(\cU_\alpha).
\end{align*}
These are $\pi$- and $\phi$-projectable, respectively, with $(\pi_*\vec v_\alpha)(0) = (\phi_*\vec w_\alpha)(0) = v_0$. By the \v Cech cohomological description of $\kappa(\pi)$ and $\kappa(\phi)$, see Section~\ref{sec:deformation theory},
\begin{align*}
\kappa(\pi)(v_0) = \left[ \big( \vec v_\alpha|_{V_{\alpha\beta}} - \vec v_\beta|_{V_{\alpha\beta}} \big)_{\alpha, \beta} \right] & \in \check H^1 \big( \{ V_\alpha \}, \sT_Y \big) = H^1( Y, \sT_Y) \quad\text{and} \\
\kappa(\phi)(v_0) = \left[ \big( \vec w_\alpha|_{U_{\alpha\beta}} - \vec w_\beta|_{U_{\alpha\beta}} \big)_{\alpha, \beta} \right] & \in \check H^1 \big( \{ U_\alpha \}, \sT_X \big) = H^1( X, \sT_X).
\end{align*}
But $\vec w_\alpha = \frac{1}{|G|} \tr(\vec v_\alpha)$. Indeed, it suffices to check this on the maximal open subset of $\cU_\alpha$ over which the quotient map $q|_{\cV_\alpha}\!: \cV_\alpha \to \cU_\alpha$ is \'etale, and there it follows from the description~\eqref{eqn:tr} of the trace map. Thus $\kappa(\phi)(v_0) = \frac{1}{|G|} \tr \big( \kappa(\pi)(v_0) \big)$. Since $v_0 \in T_0 \Delta$ was chosen arbitrarily, we see that the diagram~\eqref{eqn:KS} commutes.

\subsubsection*{Step 4: Miniversality of $\phi$}
If $G$ acts freely in codimension one, then the map $\frac{1}{|G|} \tr$ in~\eqref{eqn:KS} is an isomorphism by Lemma~\ref{lem:trace}. In particular, $\kappa(\phi)$ is surjective.
The deformation $\phi$ induces a (non-unique) map $\mu\!: \Delta \to \Deflt(X)$, whose differential at $0 \in \Delta$ is exactly $\kappa(\phi)$.
Since $\Delta$ is smooth, this shows the first inequality in the following chain:
\[ \dim_0 \Deflt(X) \ge \dim \Delta = \dim_\C T_0 \Delta = \dim_\C T_0 \Deflt(X) \ge \dim_0 \Deflt(X). \]
So $\Deflt(X)$ is smooth, $\mu$ is an isomorphism, and $\phi$ is the miniversal locally trivial deformation of $X$. \qed

\begin{lem} \label{lem:quotient basis}
Let $Y$ be a Hausdorff topological space, $G$ a finite group acting on $Y$, and $q\!: Y \to X = Y/G$ the quotient map.
Then the connected components of the sets $q^{-1}(V)$, for $V \subset X$ open, form a basis for the topology of $Y$.
\end{lem}

\begin{proof}
Pick an open set $U \subset Y$ and a point $y \in U$.
We need to find an open set $q(y) \in V \subset X$ such that the connected component of $q^{-1}(V)$ containing $y$ is contained in $U$.

We may shrink $U$ around $y$, hence we may assume that $U$ is connected.
As $Y$ is Hausdorff, we may furthermore assume that $U \cap gU = \emptyset$ for all $g \in G$ not contained in $G_y$, the stabilizer of $y$.
Replacing $U$ by $\bigcap_{g \in G_y} gU$, we may assume that $U$ is stabilized by $G_y$.
Consider $V := q(U) \subset X$.
Then $q^{-1}(V) = \bigcup_{g \in G} gU = U \cup \bigcup_{g \not\in G_y} gU$, so $V$ is open and the connected component of $q^{-1}(V)$ containing $y$ is exactly $U$.
\end{proof}

\section{An approximation criterion for quotients} \label{sec:approx crit quotient}

In this section, we prove Theorem~\ref{thm:approx crit quotient intro} from the introduction.

\subsection{Variations of Hodge structures}

This preliminary section essentially consists of a few remarks on variations of Hodge structures in the presence of a finite group action, for which we could not find a reference.

\begin{lem} \label{lem:G VHS}
Let $\pi\!: \mf Y \to B$ be a proper submersion of complex manifolds whose fibres are \kahler, and suppose that a finite group $G$ is acting on $\mf Y$ in a fibre-preserving way.
Then for any $k \in \N$, the sheaf $(R^k \pi_* \Q)^G$ is locally constant and supports a variation of Hodge structure (VHS) on $B$, which is a sub-VHS of the natural VHS on $R^k \pi_* \Q$.
\end{lem}

\begin{proof}
Any $g \in G$ induces $g^*\!: R^k \pi_* \Q \to R^k \pi_* \Q$, a morphism of VHS.
The sheaf $(R^k \pi_* \Q)^G$ is the kernel of the following morphism of VHS:
\[ \bigoplus_{g \in G} (g^* - \id)\!: R^k \pi_* \Q \to (R^k \pi_* \Q)^{\oplus |G|}. \]
Hence it is a sub-VHS of $R^k \pi_* \Q$.
\end{proof}

\begin{rem} \label{rem:G VHS}
In Lemma~\ref{lem:G VHS}, the stalk of $(R^k \pi_* \Q)^G$ at a point $t \in B$ is $H^k(Y_t, \Q)^G$, and the fibre of $(R^k \pi_* \Q)^G \tensor_\Q \O_B$ over $t$ is $H^k(Y_t, \C)^G$.
The filtration $F^\bullet$ on $H^k(Y_t, \C)^G$ induced by the Hodge filtration on $(R^k \pi_* \Q)^G \tensor_\Q \O_B$ satisfies
\[ F^p \left( H^k(Y_t, \C)^G \right) = \bigoplus_{i \ge p} H^{i,k-i}(Y_t)^G, \]
i.e.~it is exactly the natural Hodge filtration on $H^k(Y_t, \Q)^G$.
In particular, if $k = 2$, it satisfies
\[ F^2 \oplus \overline{F^1} = H^2(Y_t, \C)^G. \]
This property is important for applying~\cite[Prop.~5.20]{Voi03}, which we will do in a minute.
\end{rem}

\subsection{Proof of Theorem~\ref{thm:approx crit quotient intro}}

We may assume that $G$ acts effectively on $Y$.
Let $\pi\!: \mf Y \to \Delta$ be a deformation of $Y$ as given by Proposition~\ref{prp:extG}.
By Lemma~\ref{lem:G VHS}, the sheaf $(R^2 \pi_* \Q)^G$ supports a variation of Hodge structures, which we denote by $(H^G, F^\bullet \cH^G, \nabla^G)$.
Since the class of $\omega$ is $G$-invariant by assumption, the connection $\nabla^G$ on $\cH^G$ induces a map
\[ \overline{\nabla^G}_0^{1,1}(\omega)\!: T_0 \Delta \lto H^2(Y, \O_Y)^G \]
as in~\cite[Sec.~10.2.2]{Voi02}.

\begin{clm} \label{clm:nablaG}
The map $\overline{\nabla^G}_0^{1,1}(\omega)$ is surjective.
\end{clm}

\begin{proof}[Proof of Claim~\ref{clm:nablaG}]
The Gau\ss--Manin connection $\nabla$ on $R^2 \pi_* \Q$ induces a map
\[ \overline\nabla_0^{1,1}(\omega)\!: T_0 \Delta \lto H^2(Y, \O_Y), \]
and since $H^G$ is a sub-VHS of $R^2 \pi_* \Q$, it is clear that $\overline\nabla_0^{1,1}(\omega) = \iota \circ \overline{\nabla^G}_0^{1,1}(\omega)$, where $\iota\!: H^2(Y, \O_Y)^G \inj H^2(Y, \O_Y)$ is the inclusion.
By~\cite[Thm.~10.21]{Voi02}, we have $\overline\nabla_0^{1,1}(\omega) = \alpha \circ \kappa(\pi)$, where $\kappa(\pi)\!: T_0 \Delta \to H^1(Y, \sT_Y)$ is the Kodaira--Spencer map.
Taking $G$-invariants yields a map
\[ \alpha^G\!: H^1(Y, \sT_Y)^G \xrightarrow{[\omega] \cup -} H^2(Y, \Omega_Y^1 \tensor \sT_Y)^G \xrightarrow{\text{contraction}} H^2(Y, \O_Y)^G, \]
which is surjective since taking invariants under a finite group is an exact functor~\cite{Mas99}.
Since $\kappa(T_0 \Delta) = H^1(Y, \sT_Y)^G$, we obtain that $\overline{\nabla^G}_0^{1,1}(\omega) = \alpha^G \circ \kappa(\pi)$ is surjective, as desired.
\end{proof}

Let $\big( \cH^G \big)^{1,1}_\R$ be the real sub-vector bundle of $\cH^G$ whose fibre over $t \in \Delta$ is $H^{1,1}(Y_t)^G \cap H^2(Y_t, \R)^G \subset H^2(Y_t, \C)^G$.
Pick a neighborhood $[\omega] \in U \subset \big( \cH^G \big)^{1,1}_\R$ such that any class in $U$, lying over say $t \in \Delta$, is a \kahler class on $Y_t$.
By Claim~\ref{clm:nablaG} and Remark~\ref{rem:G VHS}, we may apply~\cite[Prop.~5.20]{Voi03} to $H^G$ and conclude that the set
\[ \Delta_\Q := \{ t \in \Delta \;|\; U \cap H^2(Y_t, \Q)^G \ne \emptyset \} \subset \Delta \]
is dense near $0 \in \Delta$, i.e.~its closure $\overline{\Delta_\Q}$ contains an open neighborhood of $0 \in \Delta$.
Now look at the deformation $\phi\!: \mf X := \mf Y/G \to \Delta$ of $X$.
By Proposition~\ref{prp:extG2}, $\phi$ is a locally trivial deformation of $X$, and for~(\ref{thm:approx crit quotient intro}.\ref{itm:acqi.1}) it remains to be proven that $\phi$ is a strong algebraic approximation.

It is sufficient to show that $\Delta_\Q \subset \Delta_\phi^{\mathrm{alg}}$ in the sense of Definition~\ref{dfn:alg approx}.
If $t \in \Delta_\Q$, then $Y_t$ carries a rational \kahler class and hence it is projective by the Lefschetz theorem on $(1, 1)$-classes and the Kodaira embedding theorem~\cite[Thm.~7.11]{Voi02}.
Then also $X_t = Y_t/G$ is projective by~\cite[Ch.~IV, Prop.~1.5]{Knu71}, so $t \in \Delta_\phi^{\mathrm{alg}}$.

If the $G$-action on $Y$ is free in codimension one, then Proposition~\ref{prp:extG2} shows that $\Deflt(X)$ is smooth and $\phi$ is the miniversal locally trivial deformation of $X$. We have just seen that $\phi$ is a strong algebraic approximation of $X$, hence~(\ref{thm:approx crit quotient intro}.\ref{itm:acqi.2}) is also proved. \qed

\section{Proof of main results} \label{sec:main results}

In this section, we will prove~(\ref{thm:approx intro}.\ref{itm:approx intro.2}), (\ref{thm:approx intro}.\ref{itm:approx intro.1}), and Theorem~\ref{thm:bim approx intro}, in this order.

\subsection{Auxiliary results}

Theorem~\ref{thm:approx intro}.\ref{itm:approx intro.1} is reduced to~(\ref{thm:approx intro}.\ref{itm:approx intro.2}) by means of the following construction.

\begin{prp}[Crepant terminalizations] \label{prp:crep term}
Let $X$ be a normal compact \kahler threefold with canonical singularities. Then there exists a projective bimeromorphic morphism $f\!: Y \to X$ from a normal compact \kahler threefold $Y$ with \Q-factorial terminal singularities such that
\[ K_Y \sim_\Q f^* K_X. \]
We say that $f\!: Y \to X$ is a crepant terminalization of $X$.
\end{prp}

\begin{proof}
By~\cite[Thm.~2.3]{CHP15}, there exists a projective bimeromorphic morphism $f\!: Y \to X$ from a normal compact \kahler threefold $Y$ with \Q-factorial terminal singularities and an $f$-exceptional boundary divisor $\Delta_Y$ on $Y$ such that $K_Y + \Delta_Y \sim_\Q f^* K_X$. Here by a boundary divisor we mean a Weil divisor with coefficients in $[0, 1] \cap \Q$.
On the other hand, we have $\Delta_Y \le 0$ by the definition of canonical singularities. Hence $\Delta_Y = 0$, and the proposition is proved.
\end{proof}

\begin{prp}[Abundance for \cy threefolds] \label{prp:cy3 abundance}
Let $X$ be a \cy threefold. Then the canonical sheaf $\omega_X$ of $X$ is torsion, meaning that some reflexive tensor power is trivial: $\omega_X^{[m]} \isom \O_X$ for some $m > 0$. In particular, $\kappa(X) = 0$.
\end{prp}

\begin{proof}
Let $f\!: Y \to X$ be a crepant terminalization of $X$. Then $c_1(Y) = 0$, in particular $K_Y$ is nef. By~\cite[Thm.~1.1]{CHP15}, the divisor $K_Y$ is semiample, i.e.~some positive multiple $mK_Y$ is globally generated (in particular, effective).
But then $mK_Y \sim 0$ since $c_1(Y) = 0$ and $Y$ is \kahler.
This implies $mK_X = f_*(mK_Y) \sim 0$, so $\omega_X^{[m]} = \O_X(mK_X) \isom \O_X$.
\end{proof}

The following general result was already mentioned in the introduction.

\begin{lem}[Stability of \cy spaces] \label{lem:cy deformation}
Let $X$ be a \cy space in the sense of Definition~\ref{dfn:cy}.
If $\pi\!: \mf X \to B$ is a deformation of $X = X_0$, then $X_t$ is again a \cy space for $t \in B$ sufficiently close to $0$.
\end{lem}

\begin{proof}
In what follows, we shrink $B$ around $0$ if necessary without further mention.

Since $X$ is \kahler and has rational singularities, also $X_t$ is \kahler by~\cite[Prop.~5]{Nam01}.
Since $X$ has canonical singularities, so do $\mf X$ and $X_t$ thanks to~\cite[Main~Theorem]{Kaw99}.
In particular, $\omega_{\mf X}$ is \Q-Cartier, i.e.~$\omega_{\mf X}^{[m]}$ is a line bundle for suitable $m > 0$.
The restriction of $\omega_{\mf X}^{[m]}$ to $X_t$ is $\omega_{X_t}^{[m]}$, for any $t \in B$, since this holds on the smooth locus of the normal space $X_t$.
The first Chern class of $\omega_{\mf X}^{[m]}$ then defines a section $s \in H^0(B, R^2 \pi_* \Z)$ whose germ at $t \in B$ is
\[ s_t = -mc_1(X_t) \in (R^2 \pi_* \Z)_t = H^2(X_t, \Z). \]
(Use~\cite[Ch.~IV, Eqn.~(13.6$'$) on p.~232]{Dem12} for the last equality.)
So we see that $c_1(X_0) = 0$ implies $c_1(X_t) = 0$ for all $t \in B$.
\end{proof}

\subsection{Proof of~(\ref{thm:approx intro}.\ref{itm:approx intro.2})}

First, we will assume that $X$ is non-projective.
Let $X_1 \to X$ be the index one cover of $X$~\cite[Def.~5.19]{KM98}. Then $\omega_{X_1} \isom \O_{X_1}$, and $X_1$ has isolated canonical singularities by~\cite[Prop.~5.20]{KM98}. Furthermore there is a finite (cyclic) group $G$ acting on $X_1$ freely in codimension one such that $X = X_1 / G$.

By~\cite[Ch.~IV, Prop.~1.5]{Knu71}, $X_1$ is non-projective since $X$ is.
Hence $X_1$ is smooth by Theorem~\ref{thm:cy3 K=0}.
\PreprintAndPublication
{By the Bogomolov--Tian--Todorov theorem~\cite[Cor.~2]{Ran92}, the deformations of $X_1$ are unobstructed.
Let $[\omega] \in H^{1,1}(X_1)$ be any \kahler class, and consider the diagram
\[ \xymatrix{
H^1(X_1, \Omega_{X_1}^2) \ar^-{[\omega] \cup -}[rr] & & H^2(X_1, \omega_{X_1}) \\
H^1(X_1, \sT_{X_1}) \ar^-{\alpha_{[\omega]}}[rr] \ar_-{\wr}[u] & & H^2(X_1, \O_{X_1}) \ar_-{\wr}[u] \\
} \]
where the vertical maps depend on the choice of an isomorphism $\omega_{X_1} \isom \O_{X_1}$, and $\alpha_{[\omega]}$ is the map defined before Theorem~\ref{thm:approx crit quotient intro}.
It is elementary to see that this diagram commutes.
By the Hard Lefschetz theorem~\cite[Thm.~6.25]{Voi02}, the upper horizontal map is surjective, and then so is $\alpha_{[\omega]}$.

Replacing $[\omega]$ by the averaged sum $\frac{1}{|G|} \sum_{g \in G} g^*[\omega]$, we may assume that $[\omega]$ is $G$-invariant.
Apply~(\ref{thm:approx crit quotient intro}.\ref{itm:acqi.2}) to conclude that the miniversal locally trivial deformation $\sX' \to \Deflt(X)$ of $X$ is a strong algebraic approximation of $X$.}
{By the Bogomolov--Tian--Todorov theorem~\cite[Cor.~2]{Ran92} and the Hard Lefschetz theorem~\cite[Thm.~6.25]{Voi02}, $X_1$ satisfies the assumptions of~(\ref{thm:approx crit quotient intro}.\ref{itm:acqi.2}).
We conclude that the miniversal locally trivial deformation $\sX' \to \Deflt(X)$ of $X$ is a strong algebraic approximation of $X$.}

Now consider the general case, i.e.~$X$ is not necessarily non-projective.
We need to show that $\overline{\Deflt(X)^{\mathrm{alg}}} = \Deflt(X)$.
So let $t \in \Deflt(X)$ be arbitrary. By Lemma~\ref{lem:cy deformation}, the fibre $X'_t$ over $t$ is again a \cy threefold with isolated singularities.
We may assume that $X'_t$ is non-projective, in which case we already know that $X'_t$ admits a locally trivial algebraic approximation.
But the deformation $\sX' \to (\Deflt(X), t)$ is a versal locally trivial deformation of $X'_t$, by ``openness of versality''~\cite[Rem.~5.8]{FK87}.
Hence $t \in \overline{\Deflt(X)^{\mathrm{alg}}}$, which implies the claim. \qed

\subsection{Proof of Theorem~(\ref{thm:approx intro}.\ref{itm:approx intro.1})}

By Proposition~\ref{prp:crep term}, there exists a crepant terminalization $f\!: Y \to X$ of $X$. Then $Y$ is a normal compact \kahler threefold with $c_1(Y) = 0$ and terminal, in particular isolated and canonical singularities~\cite[Cor.~5.18]{KM98}.
By~(\ref{thm:approx intro}.\ref{itm:approx intro.2}), $Y$ admits an algebraic approximation $\mf Y \to B$ over some complex space $0 \in B$.

\PreprintAndPublication{
Let $[X] \in \Def(X)$ and $[Y] \in \Def(Y)$ denote the miniversal deformation spaces of $X$ and $Y$, respectively, together with the miniversal deformations $\sX \to \Def(X)$ and $\sY \to \Def(Y)$.
We have $f_* \O_Y = \O_X$ since $X$ is normal, and $R^1 f_* \O_Y = 0$ since $X$ has rational singularities.
Thanks to~\cite[Prop.~11.4.1]{KM92} and the universal property of $\Def(Y)$, there exists a commutative diagram
\[ \xymatrix{
\mf Y \ar[r] \ar[d] & \sY \ar^\sF[r] \ar[d] & \sX \ar[d] \\
B \ar^-\mu[r] & \Def(Y) \ar^F[r] & \Def(X),
} \]
where the square on the left is a fibre product, $F[Y] = [X]$, and $\sF|_Y = f$.
Pick a sequence $(t_n) \subset B$ converging to $0$ such that $\mf Y_{t_n}$ is projective for all $n \in \N$. By continuity, the sequence $(s_n) \subset \Def(X)$ defined by $s_n = F(\mu(t_n))$ converges to $[X]$.
For $n \gg 0$ sufficiently large, the map $\mf Y_{t_n} \to \sX_{s_n}$ induced by $\sF$ is bimeromorphic since $f$ is. In particular, $\sX_{s_n}$ is Moishezon.
Furthermore $\sX_{s_n}$ is \kahler with rational singularities by Lemma~\ref{lem:cy deformation}.
Then $\sX_{s_n}$ is projective thanks to~\cite[Thm.~1.6]{Nam02}.
So $\sX \to \Def(X)$ is an algebraic approximation of $X$, as desired. \par}
{By~\cite[Prop.~11.4.1]{KM92}, we see that the miniversal deformation $\sX \to \Def(X)$ is an algebraic approximation.}
Now arguing exactly as before, using Lemma~\ref{lem:cy deformation} and openness of versality, we get that $\sX \to \Def(X)$ is in fact a strong algebraic approximation of $X$. \qed

\subsection{Proof of Theorem~\ref{thm:bim approx intro}}

Let $X$ be a (smooth) compact \kahler threefold with $\kappa(X) = 0$.
By~\cite[Thm.~1.1]{HP15}, there exists a bimeromorphic map $X \dashrightarrow X'$, where $X'$ is a normal \Q-factorial compact \kahler space with terminal singularities such that $K_{X'}$ is nef.
By~\cite[Thm.~1.1]{CHP15}, the divisor $K_{X'}$ is semiample, i.e.~some positive multiple $mK_{X'}$ is globally generated.
On the other hand, we have $h^0(X', mK_{X'}) \le 1$ since $\kappa(X') = \kappa(X) = 0$. Hence $mK_{X'} \sim 0$, in particular $X'$ is a \cy threefold with isolated canonical singularities.
By~(\ref{thm:approx intro}.\ref{itm:approx intro.2}), the miniversal locally trivial deformation of $X'$ is a strong algebraic approximation. \qed

\section{Fundamental groups of \kahler threefolds} \label{sec:kahler groups}

\subsection{Proof of Corollary~\ref{cor:kahler groups}}

Let $X$ be an arbitrary compact \kahler threefold. If $X$ is uniruled, consider the maximal rationally connected (MRC) quotient $\phi\!: X \dashrightarrow Z$ of $X$.
This is an almost holomorphic map with rationally connected general fibre onto a compact complex class $\cC$ manifold $Z$ of dimension $\le 2$.
Blowing up $X$ and $Z$, we may assume that $Z$ is \kahler and $\phi$ is a morphism. Then $\phi_*\!: \pi_1(X) \to \pi_1(Z)$ is an isomorphism thanks to~\cite[Thm.~5.2]{Kol93}.
The manifold $Z$, being of dimension $\le 2$, has an algebraic approximation. In particular, $\pi_1(Z)$ is a complex-projective group and then so is $\pi_1(X)$.

We may thus assume from now on that $X$ is not uniruled. In this case, we will need the following notion.

\begin{dfn}[Non-reduced locus]
Let $f\!: X \to Y$ be a surjective morphism of complex manifolds. We define the \emph{non-reduced locus of $f$} to be the subset
\[ N(f) := \big\{ y \in Y \;\big|\; \text{the fibre $f^{-1}(y)$ is everywhere non-reduced} \big\}. \]
We say that $f$ is \emph{reduced in codimension one} if $N(f)$ is contained in a codimension $\ge 2$ analytic subset $Z \subset Y$.
\end{dfn}

By~\cite[Thm.~1.1]{HP15} and~\cite[Thm.~1.1]{CHP15} there exists a bimeromorphic map $X \dashrightarrow X'$ such that $mK_{X'}$ is globally generated for some $m > 0$. Let $\phi_{|mK_{X'}|}\!: X' \to Z$ be the corresponding pluricanonical map. By construction $Z$ is a normal projective variety with $\dim Z = \kappa(X)$ and the general fibre $F'$ of $\phi_{|mK_{X'}|}$ is connected, with terminal singularities and torsion canonical divisor $K_{F'} \sim_\Q 0$.

We have an induced meromorphic map $\phi\!: X \dashrightarrow Z$. After blowing up $X$ and base changing to a suitable resolution $Z' \to Z$, we may assume that $\phi$ is holomorphic, $Z$ is smooth and the non-reduced locus of $\phi$ is contained in a simple normal crossings (snc) divisor $D = \sum_{i=1}^n D_i \subset Z$. The general fibre $F$ of $\phi$ is smooth and has Kodaira dimension $\kappa(F) = 0$.

\begin{clm} \label{clm:base change}
There is a finite surjective Galois morphism $p\!: \wt Z \to Z$ from a normal projective variety $\wt Z$ with klt singularities such that if $\wt X$ is the normalization of the main component of $X \x_Z \wt Z$, then $\wt\phi\!: \wt X \to \wt Z$ is reduced in codimension one.
\end{clm}

\begin{proof}[Proof of Claim~\ref{clm:base change}]
\PreprintAndPublication{
For every $1 \le i \le n$, pick a component $E_i \subset f^* D_i$ which surjects onto $D_i$, and let $m_i$ be its coefficient in $f^* D_i$. Put $m := m_1 \cdots m_n$.
Let $A$ be an ample divisor on $Z$ such that $mA - D$ is very ample, and pick a general element $H \in |mA - D|$.
Let $p\!: \wt Z \to Z$ be the ramified cyclic cover associated to $D + H$ as in~\cite[Def.~2.50]{KM98}.
The variety $\wt Z$ is normal, and by the ramification formula for finite maps
\[ K_{\wt Z} = p^* \left( K_Z + \frac{m-1}{m} \big( D + H \big) \right). \]
Since $D + H$ is a reduced snc divisor on $Z$, the pair $\big( Z, \frac{m-1}{m}(D + H) \big)$ is klt. Then so is $\wt Z$ by~\cite[Prop.~5.20]{KM98}.

To see that $\wt\phi$ is reduced in codimension one, fix an index $1 \le i \le n$ and a general point $P \in E_i$.
In suitable local coordinates $(x_i)$ on $X$ near $P$, $(s_i)$ on $Z$ near $\phi(P)$, and $(t_i)$ on $\wt Z$ near $\wt P := p^{-1}(\phi(P))$, the relevant data is given by
\begin{align*}
D_i & = \{ s_1 = 0 \} \subset Z, \\
\phi(x_1, \dots, x_3) & = \big( x_1^{m_i}, \phi_2(x_1, \dots, x_3) \big), \text{ and} \\
p(t_1, \dots, t_k) & = (t_1^m, t_2, \dots, t_k),
\end{align*}
where $k = \dim Z$. A calculation (whose details we omit here) now shows that $\wt\phi^{-1}(\wt P)$ contains at least one smooth point. This means that the irreducible divisor $p^{-1}(D_i)$ is not contained in $N(\wt\phi)$. A similar argument shows that also $p^{-1}(H) \not\subset N(\wt\phi)$. So $\wt\phi$ is reduced in codimension one.}
{This is a standard application of the ramified cyclic covering trick~\cite[Def.~2.50]{KM98} followed by a tedious but straightforward calculation in local coordinates, which we omit here.}
\end{proof}

The map $\wt\phi$ induces an exact sequence
\[ \pi_1(F) \lto \pi_1(\wt X) \lto \pi_1(\wt Z) \lto 1 \]
thanks to~\cite[Lemma~1.5.C]{Nor83}, and the image of $\pi_1(\wt X)$ in $\pi_1(X)$ is a normal subgroup of finite index by~\cite[Prop.~2.9.1]{Kol93}.
It remains to show that $\pi_1(F)$ and $\pi_1(\wt Z)$ are complex-projective groups.
If $\kappa(X) \ge 1$, then $\dim F \le 2$ and $\pi_1(F)$ is projective. If $\kappa(X) = 0$ then $F = X$, and $\pi_1(X)$ is projective by Corollary~\ref{cor:pi1}.

For $\pi_1(\wt Z)$, let $\wt Z' \to \wt Z$ be a projective resolution of $\wt Z$.
As $\wt Z$ is projective with klt singularities, $\pi_1(\wt Z) = \pi_1(\wt Z')$ is projective by~\cite[Cor.~1.1(1)]{Tak03}. \qed

\subsection{Proof of Corollary~\ref{cor:pi1}}

The proof proceeds in \PreprintAndPublication{four}{two} small steps.

\subsubsection*{Step 1: Projectivity of $\pi_1(X)$}

Let $X$ be a compact \kahler threefold with $\kappa(X) = 0$. By Theorem~\ref{thm:bim approx intro}, there exists a bimeromorphic modification $X \dashrightarrow X'$, where $X'$ is a compact \kahler space with terminal singularities such that $X'$ admits a locally trivial strong algebraic approximation $\mf X' \to B$.
For suitable $t \in B$, the fibre $X_t'$ will be a projective \cy threefold by Lemma~\ref{lem:cy deformation}.
Let $Y$ be a projective resolution of singularities of such an $X_t'$. Then $Y$ is a smooth projective threefold of Kodaira dimension zero, and we have a chain of isomorphisms
\begin{align*}
\pi_1(X) & \isom \pi_1(X') && \text{\cite[Cor.~1.1(1)]{Tak03}} \\
  & \isom \pi_1(X_t') && \text{Thom's first isotopy lemma, \cite[Prop.~11.1]{Mat12}} \\
  & \isom \pi_1(Y) && \text{\cite[Cor.~1.1(1)]{Tak03}.}
\end{align*}
In our situation, Thom's first isotopy lemma says that the deformation $\mf X' \to B$ is topologically trivial (possibly after shrinking $B$).

\PreprintAndPublication{
\subsubsection*{Step 2: Almost abelianity of  $\pi_1(X)$}

By~\cite[(4.17.3)]{Kol95} (see also~\cite[Thm.~1.3 and the proof of Cor.~1.4]{NS95}) the fundamental group of $Y$ has a finite index abelian subgroup $\Gamma$. Hence the same is true of $\pi_1(X)$.
There is a finite \'etale cover $\wt X \to X$ such that $\Gamma = \pi_1(\wt X) \subset \pi_1(X)$.
Taking the Galois closure of $\wt X \to X$, we may assume that $\Gamma \subset \pi_1(X)$ is a normal subgroup. Since $\rk(\Gamma)$ is the first Betti number of the compact \kahler manifold $\wt X$, it is even and we may set $r(X) := \rk(\Gamma)/2$.

\subsubsection*{Step 3: Well-definedness of $r(X)$}

Let $\Gamma_i \subset \pi_1(X)$, $i = 1, 2$, be two finite index abelian subgroups. They correspond to finite \'etale covers $X_i \to X$. The fundamental group of the fibre product $X_1 \x_X X_2 \to X$ is a subgroup $\Gamma' \subset \Gamma_1 \cap \Gamma_2$ of finite index in $\pi_1(X)$.
So $\rk(\Gamma_1) = \rk(\Gamma') = \rk(\Gamma_2)$. We see that $r(X)$ is independent of the choice of $\Gamma$.

\subsubsection*{Step 4: The augmented irregularity of $X$}

Let $\Gamma \subset \pi_1(X)$ be the finite index abelian subgroup corresponding to an \'etale cover $\wt X \to X$ with $q(\wt X) = \wt q(X)$. Then
\[ r(X) = \frac 12 \rk(\Gamma) = \frac 12 b_1(\wt X) = q(\wt X) = \wt q(X). \]
We have already noted in Remark~\ref{rem:wt q finite} that $0 \le \wt q(X) \le 3$.

The last statement of Corollary~\ref{cor:pi1} is clear since a finitely generated abelian group of rank zero is finite. \qed}
{\subsubsection*{Step 2: Remaining statements}

By~\cite[(4.17.3)]{Kol95} (see also~\cite[Thm.~1.3 and the proof of Cor.~1.4]{NS95}) the fundamental group of $Y$ has a finite index abelian subgroup $\Gamma$. Hence the same is true of $\pi_1(X)$.
The proof of the remaining statements is fairly standard, so we omit it. \qed}

\subsection{Proof of Corollary~\ref{cor:pi1 cy3}}

Passing to a crepant terminalization changes neither the fundamental group nor the holomorphic Euler characteristic. Hence we may assume that $X$ has terminal singularities.
By Theorem~\ref{thm:approx intro}, there is a locally trivial strong algebraic approximation $\mf X \to B$ of $X$.
The Euler characteristic $\chi(X_t, \O_{X_t})$ is locally constant on $B$ thanks to~\cite[Satz 3.iii)]{BPS80}. Since $\pi_1(X) = \pi_1(X_t)$ by Thom's first isotopy lemma again, we conclude by applying~\cite[Prop.~8.23]{BBdecomp} to a suitable projective fibre $X_t$. \qed \\

The following was claimed in the introduction.

\begin{lem} \label{lem:chi ne 0}
For a non-smooth, non-projective \cy threefold $X$ with $\omega_X \not\isom \O_X$, we have $\chi(X, \O_X) \ge 1$.
\end{lem}

\begin{proof}
We have $q(X) \le 1$ by Proposition~\ref{prp:q le 1} and $h^2(X, \O_X) \ge 1$ by Proposition~\ref{prp:Kodaira}.
For a projective resolution of singularities $\wt X \to X$,
\begin{align*}
h^3(X, \O_X) & = h^3 \big( \wt X, \O_{\wt X} \big) && \text{$X$ has rational singularities, \cite[Thm.~5.22]{KM98}} \\
  & = h^0 \big( \wt X, \omega_{\wt X} \big) && \text{Dolbeault isomorphism and conjugation on $\wt X$} \\
  & = h^0(X, \omega_X) && \text{$X$ has canonical singularities.}
\end{align*}
Now $h^0(X, \omega_X) = 0$ since $\omega_X$ is torsion but non-trivial. Hence
\[ \chi(X, \O_X) = 1 - q(X) + h^2(X, \O_X) - h^3(X, \O_X) \ge 1 - 1 + 1 - 0 = 1. \qedhere \]
\end{proof}

\PreprintAndPublication{
\section{Unobstructedness and smoothability} \label{sec:unobs}

\subsection{Proof of Corollary~\ref{cor:unobs}}

Let $X$ be a \cy threefold. If $X$ is projective with terminal singularities, then $\Def(X)$ is smooth by~\cite[Thm.~A]{Nam94}. Hence we may assume that $X$ is non-projective with isolated canonical singularities. Consider the index one cover $X_1 \to X$. By Theorem~\ref{thm:cy3 K=0}, $X_1$ is smooth. Then by Theorem~\ref{thm:approx crit quotient intro}, $\Deflt(X)$ is smooth. But $\Def(X) = \Deflt(X)$ by~\cite[Thms.~1 and 3]{Sch71}.

\subsection{Proof of Corollary~\ref{cor:smoothability}}

By Theorem~\ref{thm:cy3 K=0}, if $X$ is non-projective then it is smooth. The existence of a smoothing for projective $X$ is proved in~\cite[Thm.~1.3]{NS95}.

\subsection{Proof of Corollary~\ref{cor:quot cy}}

By the argument in the proof of~(\ref{thm:approx intro}.\ref{itm:approx intro.2}), the manifold $X$ satisfies the assumptions of Theorem~\ref{thm:approx crit quotient intro}. We conclude by applying~(\ref{thm:approx crit quotient intro}.\ref{itm:acqi.2}).
}{}

\newcommand{\etalchar}[1]{$^{#1}$}
\providecommand{\bysame}{\leavevmode\hbox to3em{\hrulefill}\thinspace}
\providecommand{\MR}{\relax\ifhmode\unskip\space\fi MR}
\providecommand{\MRhref}[2]{%
  \href{http://www.ams.org/mathscinet-getitem?mr=#1}{#2}
}
\providecommand{\href}[2]{#2}

\end{document}